\documentclass[10pt]{article}
\usepackage{amsmath, amsfonts, amsthm}
\usepackage{amssymb,mathrsfs, verbatim}
\usepackage{graphicx,amstext}
\usepackage{amsmath}
\usepackage{mathtools}
\usepackage{amssymb}
\usepackage{latexsym}
\usepackage{ mathrsfs }
\usepackage{enumerate}
\usepackage{xcolor}

\usepackage[latin1]{inputenc}
\newcommand{\R}{\mathbb R}

\newcommand{\ve}{\varepsilon}
\newcommand{\vp}{\varphi}

\newcommand{\sgn}{\text{sgn}}
\newcommand{\supp}{\operatorname{supp}}

\newcommand \loc    {\text{loc}}

\newcommand{\tobo}[1]{\raisebox{-1.7ex}{$\stackrel{\textstyle {\to}}
                                                    {\scriptstyle {#1}}$}}

\numberwithin{equation}{section}

\newtheorem{theorem}{Theorem}[section]

\newtheorem{remark}{Remark}[section]
\newtheorem{lemma}{Lemma}[section]
\newtheorem{corollary}{Corollary}[section]
\newtheorem{definition}{Definition}[section]

\begin{document}
\title{Existence of Complementary and Variational Weak Solutions to
Obstacle Problems for a Quasilinear 
Wave Equation}

\author{Jo\~ao-Paulo Dias$^1$, Wladimir Neves$^2$, Jos\'e-Francisco Rodrigues$^1$}

\date{\today}

\maketitle

\footnotetext[1]{Departamento de Ci\^encias Matem\'aticas and CEMS.UL, Faculdade de Ci\^encias, Universidade de Lisboa, Portugal.

E-mail: {\sl jpdias@ciencias.ulisboa.pt } and {\sl jfrodrigues@ciencias.ulisboa.pt }}.
\footnotetext[2]{Instituto de Matem\'atica, Universidade Federal
do Rio de Janeiro, Cidade Universit\'aria 21945-970,
Rio de Janeiro, Brazil. E-mail: {\sl wladimir@im.ufrj.br}.


}

\begin{abstract}
We prove the existence of weak solutions for the one obstacle problem 
associated with a class of quasilinear wave equations in one space dimension, extending previous results 
obtained in the linear case, and we also address the two obstacles problem.
In contrast with the linear setting, for both strictly quasilinear cases we obtain continuous solutions in a 
weak complementary sense, which moreover satisfy a weak entropy condition 
in the free region where the string is not in contact with the obstacles.
We further show that, in both the one and two obstacle cases, these solutions are variational solutions in a hyperbolic sense without the viscosity term.
\end{abstract}



\maketitle

\tableofcontents

\section{Introduction and the Physical Problem}
\label{INTRO}

We study in this paper the problem of dynamical (frictionless) contact or impact of a 
quasilinear elastic infinite string with a rigid foundation, the obstacle, and in between two rigid planes with possible contact,
that is to say, two rigid obstacles. 
This is a type of problem commonly encountered in mechanics, 
physics, and financial mathematics, such as in modeling elastic membranes, fluid flow, or option pricing.

\medskip
We denote by $u(x,t)$, $-\infty<x<+\infty, t\geq 0$, the transverse displacement of the 
string and we consider the following quasilinear partial differential operator 
\begin{equation}
\label{LEq}
   S[u]= u_{tt} - \big( \sigma(u_x) \big)_x
   \end{equation}
to describe the string when vibrating freely. 
The operator $S[\cdot]$ in \eqref{LEq} can be replaced by a semilinear perturbation of the type
$$
  S_k[u]= u_{tt} - \big( \sigma(u_x) \big)_x + u^k, \quad \text{where $k$ is an odd integer}. 
$$
One recalls that, for $k= 1$, the above equation is a quasilinear Klein-Gordon type equation. 

\medskip
The equation \eqref{LEq} can be obtained following Chapter I in \cite{W},
under the unique assumption of purely transversal motion, that is, the displacement of the string is directly 
along the same line as its initial position. Therefore, no smallness assumption on the displacement 
$u$ or on the strain $u_x$ is imposed. Moreover, it is well known that solutions of the
equations for purely transversal motion of elastic strings may develop shocks, that is, discontinuities
in $u_x$ and $u_t$. 
In this equation, $\sigma: \R \to \R$ is a strictly increasing smooth function, such that $\sigma(0)= 0$, describing the Cauchy stress tensor, 
as considered in \cite{SS} and \cite{C}. We may include, as a particular important case $\sigma(\lambda)= \lambda + \lambda^k$, where $k\geq3$ is an odd integer. 
Hence we are dealing with the elastic non-linear regime, and 
the stress tensor is homogeneous, that is, $\sigma$ does not 
depend on $x$ explicitly. 

\medskip
First, we suppose that the string is constrained to remain above an obstacle,
which is usually represented by a function $\varphi(x)$, hence the constraint 
is described by $u(x,t) \geq \varphi(x)$. 
In particular, we assume that the initial displacement 
$u_0 \geq \varphi$. 
We observe that, obstacles can also depend on time, 
that is, we may have $\varphi= \varphi(x,t)$, albeit the formulation become 
more involved. 
Then, following a standard formulation 
for obstacle problems, we can write 
$$
   \min \{ S[u], u - \vp \}= 0,\qquad \text{in}\quad \R \times \R_+, 
$$
which formally means that we have 
$S[u] \geq 0$ and the support of $S[u]$ is contained in $\{u= \vp \}$, since when the string is vibrating freely 
$\{u> \vp \}$, the quasilinear wave equation holds, i.e. we should have $S[u]=0$. 
Here, for simplicity of presentation, we restrict our attention to constant obstacles $(\vp= 0)$, that is to say,
we consider an infinite string which is vibrating above a flat plane with possible contact. One recalls that the string does not impose any restriction 
with respect to bending (perfectly flexible). Moreover, we assume perfect elastic reflection when the string touch the obstacle.

\medskip
Therefore, we will study the following one obstacle string problem:
given smooth initial data  $u_0=u_0(x) \geq 0$ and $v_0=v_0(x)$ for $x\in\R$, 
find $u(x,t)$, with suitable smoothness, such that, 
in an certain weak sense, 
\begin{equation}
\label{LEqBis}
u\geq 0, \quad
   S[u] \geq 0 \quad\text{and}\quad u \, S[u]=0\quad \text{in}\quad \R \times \R_+,
\end{equation}
with the initial conditions on the displacement and the velocity
\begin{equation}
\label{Initial}
 u(x,0)= u_0(x), \quad u_t(x,0)= v_0(x) \quad \text{in} \quad\R. 
\end{equation}

We observe that, from the third condition in \eqref{LEqBis}, we expect that the weak solution solves, in the sense of distributions the equation
\begin{equation}
\label{eqOpen}
   S[u]= u_{tt} - \big( \sigma(u_x \big)_x=0 \quad \text{in} \quad \{(x,t): u(x,t)>0\}
\end{equation}
whenever the subset  $\{u>0\}$ is open, which in general is difficult to prove.

\medskip
Similarly, we also consider the two obstacles problem for $u=u(x,t)$, satisfying the bilateral constraint
\begin{equation}
    \label{-1,1}
    -1\leq u(x,t) \leq 1 \quad \text{in $\R \times \R_+$},
\end{equation}
which consists of
$$
  \max\{ (u-1), \min \{ S[u], u + 1\}= 0, \quad \text{in}\quad \R \times \R_+, 
$$
and formally means that we have the corresponding three complementary conditions 
\begin{equation}
\label{2obstaces}
    \left \{
\begin{aligned}
  S[u]& \geq 0 \quad \text{in a neighborhood of $\{u=-1\}$}, 
\\[5pt]
 S[u]& = 0 \quad \text{in $\{-1<u<1\}$}, 
\\[5pt] 
S[u]& \leq 0 \quad \text{in a neighborhood of $\{u=1\}$},  
\end{aligned}
\right .
\end{equation}
with the initial conditions \eqref{Initial}, now with $-1\leq u_0\leq 1$. However, in general, $S[u]$ 
being a measure, these formal conditions require a suitable generalised formulation.

\medskip
One recalls that, in the linear case, that is when $\sigma(\cdot)$ is a linear function, with one obstacle this is a classical problem 
studied by several authors in different settings with different assumptions since the pioneer paper \cite{AP}, 
namely in \cite{Scha}, \cite{BS}, \cite{A} and \cite{FF}, or recently for the fractional case \cite{BNO} and \cite{CR}. See also the references in these papers. We observe that the numerical method of \cite{A} also yields a complementary solution in a linear case of a bounded string with one obstacle. However, to the best of our knowledge, the nonlinear setting of this problem, is completely new and the problem with the two obstacles in the hyperbolic setting remains also unexplored in the existing literature.

\medskip
Formally, we can regard the solutions of the vibrating string with obstacles as solutions, in a certain sense, of the following equation
\begin{equation}
   u_{tt} - \big( \sigma(u_x) \big)_x+ \Phi(u)=0
   \end{equation}
where the discontinuous semilinear term $\Phi(u)$ is measure valued and has support in the contact sets, $\{u=0\}$ and $\{u=-1\}\cup\{u=1\}$ for the one and the two obstacles problems, respectively.
In order to prove existence of generalised solutions to this type of highly nonlinear hyperbolic problems, we combine the 
penalisation of the obstacles with the viscosity method of D. Serre and  J. Shearer (cf. \cite{SS} and \cite{S}, see also \cite{C} 
for an important semilinear extension) and the compensated compactness technique, initially developed by F. Murat, L. Tartar and
R. DiPerna, and the $L^p$ Young measures. We observe that, the penalised problems can be regarded as hyperbolic systems for $v=u_t$ and $w=u_x$ of the type 
\begin{equation}
\label{SystHypEq}
\left \{
\begin{aligned}
    &w_{t} - v_{x}= 0, 
\\[5pt]
    &v_{t} - \big(\sigma(w)\big)_x+F(u)= 0, &&u(x,t)= u_0(x) + \int_0^t v(x, \tau) \, d\tau,   
\end{aligned}
\right .
\end{equation}
for which in \cite{C} the existence of a weak entropy solution was proved for a monotone Lipschitz 
continuous real function $F$ with power growth and $F(0)=0$. 

\medskip
Section \ref{SMR} states the main results of this article. 
In particular, it introduces the definitions of complementary and variational weak 
solutions for obstacle problems associated with the quasilinear partial differential operator \eqref{LEq}, 
and states the existence theorems, namely Theorem \ref{MainThmSol1} and Theorem \ref{MainThmSol2}.
Lastly, this section also recalls the definition of entropy for the system \eqref{SystHypEq} and a compactness result, which provides the continuity regularity in one space dimension and a uniform convergence result of the approximating solutions (Theorem \ref{MainThmReg}). 
More precisely, we establish a refined analysis of the regularity of solutions to string obstacle problems, 
which are continuous in $[-R,R] \times ]0,T]$ 
for every $T> 0$ and $R> 0$. 
This allows us to show, for instance, 
that the set $\{u>0\}$ is open in the one obstacle string
problem, and similarly that the set $\{-1<u<1\}$ 
is open in the two obstacles case. The local continuity of the solution gives a sense to the third complementary equation in \eqref{LEqBis} and to the complementary conditions \eqref{2obstaces}, respectively, for the one and the two obstacles problems. 

\medskip
Section \ref{ProofMainThm} presents the penalisation strategy combined with the viscosity method to solve the one and two obstacle problems.  
In Section \ref{ProofMainThmSol1}, we study the one obstacle string problem \eqref{LEqBis} 
by applying the penalisation method (cf. \cite{L}, \cite{L1}) and the parabolic vanishing viscosity 
technique (cf. \cite{SS}, \cite{C}), complemented by the compensated compactness method, 
in order to obtain a weak solution to \eqref{LEqBis} with $u(x,t) \geq 0$, for initial data 
$u_0=u(0) \geq 0$ and $v_0= u_t(0)$.
We consider an approaching sequence $\{u_\ve\}$ satisfying the system \eqref{SystReg}, 
which in the limit as $\ve \to 0$ yields a complementary solution 
$u \ge 0$ under the local continuity regularity, which is established using uniform estimates and the Theorem \ref{MainThmReg}.
To show that is also a variational weak solution, we establish the new result on the strong local convergence in 
$L^2$ of the approximate velocities produced by the vanishing viscosity method, in the essential Theorem \ref{ThmStongConvL2}.

\medskip
In Section \ref{HQI} we consider the two obstacles string problem \eqref{2obstaces}.
We follow essentially the same strategy applied to the one obstacle problem in order to obtain a weak solution 
to \eqref{2obstaces} for initial data $u_0= u(0)$, satisfying
$-1 \leq u_0 \leq 1$, and $v_0= u_t(0)$.
The approaching sequence $\{u_\ve\}$ satisfying the system \eqref{Inequality}, 
with the penalisation term $F_\ve$, which yields 
$-1 \leq u \leq 1$ as $\ve \to 0$, 
unlike the one obstacle problem, 
now has no fixed sign bringing additional difficulties. In fact,  
the two obstacles complementary problem is not fully analogous to the one obstacle case, since it is a simultaneously lower and upper obstacle problem, and so the proof of its existence result requires additional arguments.

\subsection{Structural and physical hypothesis} 
\label{SPH}

Let $\sigma: \R \to \R$ be a smooth increasing function such that, $\sigma(0)= 0$. Then, we set
\begin{equation}
    \Sigma(\lambda)= \int_0^\lambda \sigma(s) ds, 
\end{equation}
and denote 
$$
   L^\Sigma(\R)= \{ f \in L^1(\R); \int_\R \Sigma(f(x)) \, dx< \infty \}. 
$$
Moreover, for $T>0$ we denote $\R_T= \R \times (0,T)$ and consider the space:  
$$
\begin{aligned}
   L^\Sigma_{\rm loc}(\R_T)&= \big\{ f \in L^1_{\rm loc}(\R_T):
   \\[5pt]
   & \iint_K \Sigma(f(x,t)) \, dxdt < \infty, \text{for any compact set $K \subset \R_T$}\big\}. 
   \end{aligned}
$$

{\bf Structural hypothesis:} Due to the strict hyperbolicity, genuine nonlinearity and suitable growth constraints will be assumed the following conditions. 
\begin{enumerate}
\item[H1.] There exists $c> 0$, such that, $\sigma^\prime(\lambda)\geq c$, for each $\lambda \in \R$; 

\smallskip
\item[H2.] For each $\lambda \in \R$, $\sigma^{\prime\prime}(\lambda) \neq 0$, or there exists $\lambda_0 \in \R$ 
such that $\sigma^{\prime\prime}(\lambda_0)= 0$, and  
$\sigma^{\prime\prime}(\lambda) \neq 0$ for any $\lambda \neq \lambda_0$;

\item[H3.] Moreover, 
$$
\frac{\sigma^{\prime\prime}}{(\sigma^\prime)^{5/4}}, \frac{\sigma^{\prime\prime\prime}}{(\sigma^\prime)^{7/4}} \in L^2(\R), \quad 
\frac{\sigma^{\prime\prime}}{(\sigma^\prime)^{3/2}}, \frac{\sigma^{\prime\prime\prime}}{(\sigma^\prime)^{2}} \in L^\infty(\R); 
$$
and there are $m>0$ and $q> 1/2$, such that 
\begin{equation}
\label{HYP1} 
    \big(\sigma^\prime(\lambda)\big)^q \leq m \, (1 + \Sigma(\lambda)). 
\end{equation} 

%
\item[H.4]
In addition, we also require there exists a constant $\kappa> 0$, such that, for any $\lambda \in \R$, 
\begin{equation}
\label{newcondsigma}
| \lambda \, \sigma(\lambda)| \leq \kappa \, \Sigma(\lambda). 
\end{equation}
In particular, for each $a, b \in \R$, it follows that  
\begin{equation}
\label{newcondsigmaGNR}
    | b \; \sigma(a)| \leq \kappa \, \big( \Sigma(a) + \Sigma(b) \big). 
\end{equation}
\end{enumerate}

\begin{remark}
We observe that, by the assumption H1 and $\sigma(0)=0$, we have 
\begin{equation}
\label{continuity}
  \Sigma(\lambda) \geq \frac{c}{2} \, |\lambda|^2, \quad \text{for all $\lambda \in \R$}, 
\end{equation} 
\begin{equation}
\label{Sigmasigma}
   \lim_{|\lambda| \to \infty} \frac{\sigma(\lambda)}{\Sigma(\lambda)}= 0,
\end{equation}
and these conditions hold for $\sigma(\cdot)$ with suitable polynomial type behaviour.
\end{remark} 

\bigskip
{\bf Initial data:} For the initial data, we suppose
\begin{equation}
    \label{initial data}
    u_0 \in H^1_\Sigma(\R) \cap H^2(\R) \quad  \text{and}\quad v_0 \in H^1(\R), 
\end{equation}
where 
$$
  H^1_\Sigma(\R)= \{ f \in L^2(\R); \, f_x \in L^\Sigma(\R) \}. 
$$
Similarly, we define $H^1_\Sigma(\R_T)$.

\section{Statement of the Main Results} 
\label{SMR}

We are concerned with obstacle problems for evolution equations in elasticity, and we shall introduce two types of weak solutions:

\begin{itemize}
    \item \textbf{Complementary weak solutions}, which are weaker in structure, but require the local continuity of the solution;
    \item \textbf{Variational weak solutions}, which follow the Brezis-Lions formulation, see \cite{L} and \cite{JMNS}, for hyperbolic variational inequalities. 
\end{itemize}

First, we present the complementary formulations for continuous functions $u=u(x,t)$, where we introduce the open sets $\{u>0\}=\{(x,t)\in \R_T: u(x,t)>0\}$ and $\{-1<u<1\}=\{(x,t)\in \R_T: -1<u(x,t)<1\}$  (similarly for $\{u<0\}$). 
\begin{definition}
\label{DefCompOne}
A continuous function $u=u(x,t)$ is called a complementary weak solution of the one obstacle problem, when $u\in C(\R_T)$ safisties for any $T> 0$, 
\begin{equation}
\label{DefCompOne01}
    (u_x, u_t) \in L^\infty(0,T;L^\Sigma(\R)\cap L^2(\R)) \times L^\infty(0,T;L^2(\R)), 
\end{equation}
\begin{equation}
\label{DefCompOne02}
    u \in L^2(\R_T), \quad u(0)= u_0 \quad \text{in} \quad L^2(\R), 
\end{equation}
\begin{equation}
\label{UPos}
\text{$u \geq 0$ \, \,  in \, \ $\R_T$,}
\end{equation}
 and for any nonnegative test function $\varphi \in C^\infty_c(\R\times[0,T))$, 
\begin{equation}
\label{SolOneWeak}
    \iint_{\R_T} (\sigma(u_x) \, \varphi_x - u_t \, \varphi_t)  \, dxdt
    \geq \int_\R v_0 \, \varphi(0) \, dx. 
\end{equation}
Moreover, the function $u$ satisfies is distribution sense
\begin{equation}
\label{eqOComOne}
   u_{tt} - \big(\sigma(u_x)\big)_x= 0 \quad \text{in} \quad \R_T\cap\{u>0\}.
\end{equation}
\end{definition}

\begin{definition}   
\label{DefCompTwo}
A continuous function $u=u(x,t)$ is called a complementary weak solution of the two obstacles problem, when $u\in C(\R_T)$ satisfies 
\begin{equation}
    (u_x, u_t) \in L^\infty(0,T;L^\Sigma(\R)\cap L^2(\R)) \times L^\infty(0,T;L^2(\R))
\end{equation}
\begin{equation}
    u \in L^2(\R_T), \quad u(0)= u_0 \quad\text{in} \quad L^2(\R), 
\end{equation}
\begin{equation}
\label{UTwo}
\text{$-1\leq u\leq 1$\, \,  in \, \,$\R_T$,} 
\end{equation}
and the function $u$ satisfies is distribution sense
\begin{equation}
\label{eqOComTwoBelow}
     u_{tt} - \big(\sigma(u_x)\big)_x \geq 0 \quad \text{in}\quad \R_T \cap\{u< 0\},
\end{equation}
\begin{equation}
\label{eqOComTwoAbove}
   u_{tt} - \big(\sigma(u_x)\big)_x \leq 0 \quad \text{in}\quad \R_T \cap\{u> 0\},
\end{equation} 
\begin{equation}
\label{eqOComTwo}
   u_{tt} - \big(\sigma(u_x)\big)_x= 0 \quad \text{in} \quad \R_T \cap\{-1<u<1\}.
\end{equation}
Moreover, for any nonnegative test function $\varphi \in C^\infty_c(\R\times[0,T))$, such that, \text{${\rm spt }\vp \cap \{t>0\} \subset \{u< 0\}$}
\begin{equation}
\label{SolTwoWeakP}
    \iint_{\R_T} (\sigma(u_x) \, \varphi_x - u_t \, \varphi_t)  \, dxdt
   \geq \int_\R v_0 \, \varphi(0) \, dx, 
\end{equation}
for any nonnegative $\varphi \in C^\infty_c(\R\times[0,T))$, such that, \text{${\rm spt }\vp \cap \{t>0\} \subset \{u> 0\}$}
\begin{equation}
\label{SolTwoWeakN}
    \iint_{\R_T} (\sigma(u_x) \, \varphi_x - u_t \, \varphi_t)  \, dxdt
   \leq \int_\R v_0 \, \varphi(0) \, dx, 
\end{equation}
 and for any $\phi \in C^\infty_c(\R\times[0,T))$ 
\begin{equation}
\label{SolTwoWeak0}
   \iint_{\R_T} (\sigma(u_x) \, \phi_x - u_t \, \phi_t)  \, dxdt
    =\int_\R v_0 \, \phi(0) \, dx. 
\end{equation}

\end{definition}

\smallskip
Now we present the weak variational formulations, which encode more than just complementarity. Indeed, they describe the full interaction 
between solution and obstacle in a convex functional framework.

\begin{definition}
\label{DefVarOne}
The function $u\in L^2(0,T;L^2(\R))$, for any $T> 0$,  is called a variational weak solution of the one obstacle problem, when it satisfies 
$$
   (u_x,u_t) \in L^\infty(0,T;L^\Sigma(\R)\cap L^2(\R)) \times L^\infty(0,T;L^2(\R)), 
$$
$$
 u\geq0\quad \text{a.e. in}\quad \R_T, \quad u(0)=u_0 \quad\text{in} \quad L^2(\R), 
 $$ 
\begin{equation}
\label{EqVarOne}
\begin{aligned}
\iint_{\R_T} \sigma(u_x) (\zeta_x-u_x) \, dxdt &- \iint_{\R_T} u_t \, (\zeta_t-u_t) \, dxdt
\\[5pt]
    &\geq \int_\R v_0(\zeta(x,0)-u_0(x)) \, dx 
\end{aligned}
\end{equation}
for all $\zeta \in H^1(0,T; L^2(\R))\bigcap L^2(0,T;H^1(\R)) \bigcap H^1_\Sigma(\R_T)$,\,
such that, $\zeta \geq 0$ and $\zeta(x,T)=0$.
\end{definition} 

\smallskip
\begin{definition}
\label{DefVarTwo}
The function $u\in L^2(0,T;L^2(\R))$, for any $T> 0$, is called a variational weak solution of the two obstacles problem, when it satisfies 
$$
   (u_x,u_t) \in L^\infty(0,T;L^\Sigma(\R)\cap L^2(\R)) \times L^\infty(0,T;L^2(\R)), 
$$
$$
 -1\leq u\leq1\quad \text{a.e. in}\quad \R_T, \quad u(0)=u_0 \quad\text{in} \quad L^2(\R), 
$$ 
\begin{equation}
\label{EqVarTwo}
\begin{aligned}
\iint_{\R_T} \sigma(u_x) (\xi_x-u_x) \, dxdt &- \iint_{\R_T} u_t \, (\xi_t-u_t) \, dxdt
\\[5pt]
    &\geq \int_\R v_0(\xi(x,0)-u_0(x)) \, dx  
\end{aligned}
\end{equation}
for all $\xi \in H^1(0,T; L^2(\R))\bigcap L^2(0,T;H^1(\R)) \bigcap  H^1_\Sigma(\R_T)$, such that 
$-1\leq\xi \leq 1$ and $\xi(x,T)=0$.
\end{definition} 

The main results of this work are the following two theorems on the existence of solutions with the regularity in $C(0,T;C^{0,\alpha}_{loc}(\R))$, the space of continuous functions from [0,T] into the space of H\"older continuous functions $C^{0,\alpha}([-R,R]))$, for all $R>0$ and $0<\alpha<1$, equipped with the usual norms. 
\smallskip
\begin{theorem}[One Obstacle String Problem]
\label{MainThmSol1}
Under the structural hypothesis H1-H3, and \eqref{initial data} for the initial conditions 
with $u_0\geq 0$ in $\R$, there exists a complementary weak solution $u\in C(0,T;C^{0,\alpha}_{loc}(\R))$, with $0<\alpha<1/2$, in the sense of 
Defintion \ref{DefCompOne}. Moreover, this solution is also a
variational weak solution as given by Definition \ref{DefVarOne}. 
\end{theorem} 

\medskip
For the two obstacles problem, we have the following

\begin{theorem}[Two Obstacles String Problem]
\label{MainThmSol2}
Under the structural hypothesis H1-H4, and \eqref{initial data} 
for the initial data where $-1\leq u_0\leq 1$ in $\R$, 
there exists a complementary weak solution $u\in C(0,T;C^{0,\alpha}_{loc}(\R))$, with $0<\alpha<1/2$, in the sense of 
Defintion \ref{DefCompTwo}. Moreover, this solution is also a
variational weak solution as given by Definition \ref{DefVarTwo}.
\end{theorem} 

\subsection{Entropy inequality and Continuity of Solutions}
\label{profMT}

\medskip
Now, we recall the definition of entropy for the system \eqref{SystHypEq}.

\begin{definition}
\label{EntropyDef}
A pair of scalar functions $(\eta, H)$ is called an entropy-entropy flux pair 
for the system \eqref{SystHypEq}, if all sufficiently-smooth solutions $(w,v)$ of \eqref{SystHypEq} also satisfy
\begin{equation}
\label{EntropyEq}
   \big(\eta(w,v)\big)_t +  \big(H(w,v)\big)_x+\nabla \eta \cdot (0,F(u))= 0. 
\end{equation} 
It is sufficient that $\eta$ and $H$ satisfy, for $f(w,v)= (-v, -\sigma(w))$, 
\begin{equation}
\label{EntropyCond}
    \nabla f(w,v) \, \nabla \eta(w,v)= \nabla H(w,v).  
\end{equation} 
Similarly, a pair $(w,v)$ is called an entropy solution of \eqref{SystHypEq}, 
when it is a weak solution satisfying, in the sense of distributions in $\R \times \R_+$, 
and for any convex entropy $\eta$ of flux $H$, 
\begin{equation}
\label{EntropyIneq}
   \big(\eta(w,v)\big)_t +  \big(H(w,v)\big)_x+\nabla \eta \cdot (0,F(u)) \leq 0. 
\end{equation} 
\end{definition}
 Here we shall consider the entropy-entropy flux pair
\begin{equation}
\label{EntropyPair}
   \eta(w,v)= \frac{v^2}{2} + \Sigma(w), \quad H(w,v)=-v \, \sigma(w),
\end{equation} 
for which we are able to show the existence of weak solutions to the hyperbolic obstacle problems, satisfying the entropy inequality
\eqref{EntropyIneq} with $F \equiv 0$ in the regions $\{u>0\}$ and $\{-1<u<1\}$ where the string vibrates with no contact, respectively in the one and the two obstacles cases.

We conclude this section by stating an auxiliary compactness result adapted to the approximating solutions to the obstacle problems and yielding the local continuity of the solutions.

\begin{theorem}
\label{MainThmReg}
Let $q>1$, and assume 
\begin{equation}
\label{UveReg}
\begin{aligned}
  M \subset Y_q = \Big\{ f \in L^{\infty}(0,T; W^{1,q}_\loc(\R)) \,|\, \, \, 
\frac{\partial f}{\partial t} \in L^2(0,T; L^{2}_\loc(\R))  \Big\}, 
\end{aligned}
\end{equation}
and $M$ is a bounded subset in $Y_q$. Then M is relatively compact in 
\begin{equation}
\label{RegulaU}
   C(0,T; C^{0,\alpha}_{loc}(\R)), \,\,\,\textit {for $0<\alpha<1-1/q$} .  
\end{equation}
In particular, if $\{u_\ve\}$ is a family of functions converging weakly to $u$ in $Y_q$, then
\begin{equation}
\label{uniformconverg}
   \text{ $u_\ve$ converges uniformly to $u$ on compact subsets of $\R \times [0,T]$.} 
\end{equation}
\end{theorem}

\begin{proof} Recalling that for any interval $[-R, R] \subset \R$, the Sobolev space $W^{1,q}(-R, R)$, for $q>1$ is compactly embedded in $C^{0,\alpha}([-R, R])$, for $0<\alpha<1-1/q$, this theorem is a direct application of Theorem 5, p. 84 of \cite{Simon}. 
\end{proof}

\section{Regularisation and Penalisation Approximation}
\label{ProofMainThm}

In the framework of the hyperbolic system of conservation laws \eqref{SystHypEq}, we shall use, respectively,
$$
    F_\ve(z)= - \frac{1}{\ve}(z^{-}) \quad \text{and} \quad F_\ve(z)= - \frac{1}{\ve}\big((z+1)^{-} - (z-1)^{+} \big)
$$    
for the penalisation of the one obstacle and the two obstacles problem. 
In both cases, we consider regularization of the initial conditions
\begin{equation}
    \label{initial regdata}
    u_0^\varepsilon \in H^3(\R) \cap H^1_\Sigma(\R) \quad  \text{and}\quad v_0^\varepsilon \in H^2(\R), 
\end{equation}
such that 
$$
   u_0^\ve \to u_0,  \quad v_0^\ve \to v_0 \quad \text{ strongly in $H^1_\Sigma(\R) \cap H^2(\R)$ and $H^1(\R)$ respectively}.
$$
Morevoer, for some fixed constant $C> 0$ , independent of $\varepsilon$,
 \begin{equation}
\label{EnergInitial}
\begin{aligned}
\frac{1}{2} \int_{\R} (v^\ve_0)^2 \, dx + \int_{\R} \Sigma(u_{0 x}^\varepsilon) \, dx \leq C, 
\\[5pt] 
 \int_{\R} (v^\ve_{0 x})^2 \, dx + \int_{\R} (u_{0 xx}^\ve)^2 \, dx \leq C. 
\end{aligned} 
\end{equation}
 
\subsection{One obstacle string problem}
\label{ProofMainThmSol1}

In this section, we present the proof of Theorem \ref{MainThmSol1}, which is carried out in several steps.
We start with the following 

\begin{theorem}
\label{MainThm} 
Let $u^\ve_0 \in H^3(\R)\cap H^1_\Sigma(\R)$ be non-negative, $v^\ve_0 \in H^2(\R)$ for any $\ve \in ]0,1]$ fixed. Under conditions $H1$-$H3$, there exists 
$$
   u_\ve \in C([0,+\infty[; H^3(\R)) \cap C^1([0,+\infty[; H^2(\R)) \cap C^2([0,+\infty[; L^2(\R)),
$$
such that $u_\ve(0)= u_0$, and the pair $(w_\ve, v_\ve)= (u_{\ve x}, u_{\ve t})$ is the unique solution of the 
Cauchy problem $(t \geq 0, x \in \R)$, 
\begin{equation}
\label{SystReg}
\left \{
\begin{aligned}
    &w_{\ve t}= v_{\ve x}, 
\\[5pt]
    &v_{\ve t} - \big(\sigma(w_\ve)\big)_x = \frac{u_\ve^-}{\ve} + \ve \, v_{\ve xx},  
    \\[5pt]
    &w_{\ve}(0)= w_0^\varepsilon= u_{0 x}^\varepsilon, \quad v_{\ve}(0)= v_0^\ve, 
\end{aligned}
\right .
\end{equation}
where $\lambda^-= (-\lambda)^+= (|\lambda| - \lambda)/ 2$.
Moreover, for each compact set $K \subset \R$, there exists $C_1> 0$ independent of $\ve> 0$, such that, 
\begin{equation}
\label{LInftyIneq} 
\frac{1}{\ve} \int_0^t \!\! \int_K u_\ve^- \, dxd\tau \leq C_1 , \qquad (\forall t \geq 0, \; \forall \ve \in ]0,1]).  
\end{equation}
\end{theorem} 

\begin{proof} 
1. First, observing that $(u^-)_x= \frac{1}{2} \,(\sgn \, u - 1) \, u_x$, 
the first part of the proof of Theorem \ref{MainThm} follows directly from 
Theorem 2.2 in \cite{C}. Moreover, upon defining 
$$
  E_\ve(t):= \frac{1}{2} \int_{\R} v^2_\ve \, dx + \int_{\R} \Sigma(w_\ve) \, dx + \frac{1}{2  \, \ve} \int_{\R} (u^-_\ve)^{2} \, dx, 
$$
we have the following energy estimate 
\begin{equation}
\label{EnergEst}
E_\ve(t) + \ve \int_0^t \!\! \int_{\R} v_{\ve x}^2 \, dx= E_\ve(0) \leq C,
\end{equation}
for each $t \geq 0$ and some positive constant $C$, where we have used \eqref{EnergInitial}. 

\medskip
2. To prove \eqref{LInftyIneq}, we fix 
$K\Subset K'\Subset \R$, where $A \Subset B$ denotes that 
the set $A$ is compactly contained in $B$.
Let $\vp \in C^\infty_c(K')$ be a cutoff function, $0 \leq \vp \leq 1$, $\vp \equiv1$ on $K$
and $|\vp_x|= O(\delta^{-1})$, for $0< \delta=\mathrm{dist}(K,\partial K')\leq 1$. 
Now, for $R> 0$ sufficiently large, and any $x_o \in K$ fixed, we take the test function 
$$\vp_R(x)= \vp(x_o + x/R).$$
Then, we multiply $(\ref{SystReg})_2$ by $\vp_R$ and integrate in $(0,t) \times \R$, for any $t>0$. 
Thus we may write     
$$
\begin{aligned}
\frac{1}{\ve} \int_0^t \int_{K'}& u_\ve^- \, \vp_R \, dx d\tau  - \int_{K'} v_{\ve}(t) \, \vp_R \ dx + \int_{K'} v_{0} \, \vp_R \ dx
\\[5pt]
&= \frac{1}{R} \int_0^t \int_{K'} \sigma(w_\ve) \, \vp_x \, dx d\tau +  \frac{\ve}{R} \int_0^t \int_{K'} v_{\ve x} \, \vp_x \, dx d\tau. 
\end{aligned}
$$
Therefore, we have
$$
\begin{aligned}
\frac{1}{\ve} \int_0^t \int_{K} u_\ve^-\, dx d\tau &\leq C \Big( \|v_\ve(t)\|_{L^2(\R)} + \|v_0\|_{L^2(\R)}
\\[5pt]
& + \frac{\delta^{-1}}{R} \int_0^t \int_{K'} |\sigma(w_\ve)| \, dx d\tau + \frac{\ve}{R} \|v_{\ve x}\|_{L^2(\R)} \Big), 
\end{aligned}
$$
for some positive constant $C$. Passing to the limit as $R \to \infty$, we obtain for each $t \geq 0$
$$
\frac{1}{\ve} \int_0^t \int_K u_\ve^- \, dx d\tau \leq C_1, 
$$
where we have used \eqref{EnergEst}, and that $\int_{K'} |\sigma(w_\ve)| dx$ is uniformly bounded. Indeed, due to \eqref{Sigmasigma}, there 
exists $\Lambda> 0$, such that, if $|w_\ve| \geq \Lambda$, then 
$$
     |\sigma(w_\ve)| \leq \Sigma(w_\ve). 
$$
Therefore, we have 
$$
\begin{aligned}
\int_{K'} |\sigma(w_\ve)| \, dx&= \int_{K' \cap \{|w_\ve| \geq \Lambda\}} |\sigma(w_\ve)| \, dx + \int_{K' \cap \{|w_\ve| < \Lambda\}} |\sigma(w_\ve)| \, dx 
\\[5pt]
&\leq \int_{\R} \Sigma(w_\ve) \, dx + |K'| \max_{|\lambda|\leq \Lambda} |\sigma(\lambda)| \leq C, 
\end{aligned}
$$
where we have used \eqref{EnergEst} and $C> 0$ is independent of $\ve>0$. 

\end{proof} 
\medskip
Now, we consider a crucial estimate for the pair $(w_\ve, v_\ve)= (u_{\ve x}, u_{\ve t})$,
the unique solution given in Theorem \ref{MainThm} for the Cauchy problem \eqref{SystReg}.
\begin{lemma}
\label{Lemma1}
Under the assumptions of Theorem \ref{MainThm} we have, for $\ve \in ]0,1]$ and with a positive constant $c_2$,
the following estimate 
\begin{equation}
\label{Lemma1Eq}
\begin{aligned}
   \ve \int_0^t \!\! \int_{\R} \sigma^\prime(w_\ve) \, (w_{\ve x})^2 \, dx d\tau &+ \frac{\ve^2}{2} \int_{\R} (w_{\ve x})^2(t) \, dx +  \int_{\R} (v_{\ve})^2(t) \, dx 
   \\[5pt]
   & +  \ve \int_0^t \!\! \int_{\R} (v_{\ve x})^2 \, dx d\tau \leq c_2 \, (1 + t), \qquad t \geq 0. 
\end{aligned}
\end{equation}
\end{lemma}

\begin{proof}
We have (suppressing the index $\ve$ to simplify): 
\begin{equation}
\label{Lemma1Proof}
\begin{aligned}
 - \frac{d}{dt} \int_{\R} v_x \, w \, dx&= - \int_{\R} v_{x t} \, w \, dx -  \int_{\R} v_x \, w_t \, dx
 = \int_{\R} v_t \, w_x \, dx - \int_{\R} v_x \, w_t \, dx
 \\[5pt]
 & = \int_{\R} v_t \, w_x \, dx -  \int_{\R} v^2_x  \, dx. 
 \end{aligned}
\end{equation}
Since $v_{xx}= w_{tx}$, we have 
$$
  \int_{\R} \big(v_t \, w_x - \sigma^\prime(w) \, (w_x)^2 \big) \, dx= \ve \, \int_{\R} v_{xx} \, w_x \, dx + \frac{1}{\ve} \int_{\R} u^- \, w_x \, dx, 
$$
thus we derive from \eqref{Lemma1Proof}, 
$$
\begin{aligned}
 - \frac{d}{dt} \int_{\R} v_x \, w \, dx &+ \int_{\R} (v_x)^2  \, dx - \int_{\R} \sigma^\prime(w) \, (w_x)^2 \, dx
 \\[5pt]
&= \int_{\R} v_t \, w_x \, dx  - \int_{\R} \sigma^\prime(w) \, (w_x)^2 \, dx
\\[5pt]
&= \ve \, \int_{\R} v_{xx} \, w_x \, dx + \frac{1}{\ve} \, \int_{\R} u^- \, w_x \, dx 
\\[5pt]
&= \frac{\ve}{2} \, \frac{d}{dt} \int_{\R} (w_x)^2 \, dx + \frac{1}{\ve} \, \int_{\R} u^- \, w_x \, dx, 
 \end{aligned}
$$
and so 
\begin{equation}
\label{Lemma1Proof10}
\begin{aligned}
   & - \! \int_{\R} (v_x \, w)(t)  dx  +\!  \int_{\R} v_{0 x} \, u_{0 x} dx + \! \int_0^t \!\! \int_{\R} (v_x)^2 dx d\tau - \! \int_0^t \!\! \int_{\R} \sigma^\prime(w) \, (w_x)^2 dxd\tau
    \\[5pt]
    & = \frac{\ve}{2} \int_{\R} (w_x)^2(t) \, dx -  \frac{\ve}{2} \int_{\R} (u_{0xx})^2 \, dx + \frac{1}{\ve}  \int_0^t \!\! \int_{\R} u^- \, w_x \, dxd\tau. 
\end{aligned}
\end{equation}

On the other hand, since $u^-= \frac{1}{2} (|u| - u)$, we have 
\begin{equation}
\label{Lemma1Proof20}
\begin{aligned}
\frac{1}{\ve}  \int_0^t \!\! \int_{\R} u^- \, w_x \, dxd\tau&= - \frac{1}{\ve} \int_0^t \!\! \int_{\R} (u^-)_x \, u_x  \, dxd\tau
\\[5pt]
&= - \frac{1}{2 \ve}  \int_0^t \!\! \int_{\R} (\sgn u) \, u_x^2  \, dxd\tau +  \frac{1}{2 \ve}  \int_0^t \!\! \int_{\R} u_x^2  \, dxd\tau \geq 0, 
\end{aligned}
\end{equation}
and so by \eqref{Lemma1Proof10}, \eqref{Lemma1Proof20}, we derive 
\begin{equation}
\label{Lemma1Proof30}
\begin{aligned}
& \int_0^t \!\! \int_{\R}  \sigma^\prime(w) \, (w_x)^2 \, dxd\tau +  \frac{\ve}{2} \int_{\R} (w_x)^2(t)  \, dx \leq \int_{\R} (v \, w_x)(t) \, dx + \int v_{0 x} \, u_{0 x} dx 
\\[5pt]
&+ \int_0^t \!\! \int_{\R} (v_x)^2  \, dxd\tau +  \frac{\ve}{2}  \int_{\R} (u_{0xx})^2  \, dx \leq  \frac{\ve}{4} \int_{\R} (w_x)^2(t)  \, dx + \frac{1}{\ve} \int_{\R} v^2(t) \, dx 
\\[5pt]
&+  \|v_{0 x}\|_2 \,\| u_{0 x}\|_2 + \int_0^t \!\! \int_{\R} (v_x)^2  \, dxd\tau. 
\end{aligned}
\end{equation}

Finally, \eqref{Lemma1Eq} is a consequence of \eqref{EnergEst} and \eqref{Lemma1Proof30}, and Lemma \ref{Lemma1} is proved. 
\end{proof} 

\medskip
Now, following \cite{SS} (see also \cite{C}), we need to consider a class of entropy-entropy flux pairs, 
defined as solutions of a suitable Goursat problem associated with an appropriate system. 
This approach is required since we do not have $L^\infty$ estimates for the approximate solutions 
$(w_\ve, v_\ve)$ in Theorem~\ref{MainThm}.
Then, we must use entropy-entropy flux pairs $(P,Q)$, such that 
$$
  \big|\frac{P}{\eta}\big|, \big|\frac{Q}{\eta}\big| \to 0 \quad \text{as $|w|+|v| \to \infty$, \; where $\eta(w,v)$ is given by \eqref{EntropyPair}}. 
$$
By the energy estimate \eqref{EnergEst} and applying the Young measure theorem (cf. \cite{B}, see also \cite{C}), we can associate 
to the sequence $\{(w_\ve, v_\ve)\}_\ve$ in Theorem \ref{MainThm} a subsequence still denoted by $\{(w_\ve, v_\ve)\}_\ve$, such that 
$$
   \lim_{\ve \to 0} (w_\ve, v_\ve)= (w, v) \in \big( L^2_{\rm loc}(\R_T) \big)^2 \quad \text{weakly.} 
$$
To prove that,  for each $\theta \in C_c^\infty(\R_T)$,
\begin{equation}
\label{ConvV}  
     \lim_{\ve \to 0} \iint_{\R_T}  \sigma(w_\ve) \, \theta_x \, dxdt=  \iint_{\R_T}  \sigma(w) \, \theta_x \, dxdt, 
\end{equation}
we need to use the estimates \eqref{Lemma1Eq} and \eqref{LInftyIneq}, to apply Murat's lemma (cf. \cite{M}) 
like in \cite{S} (cf. also \cite{SS} and \cite{C}). More precisely, let $(P,Q)$ be an entropy-entropy flux pair as described above.
Following \cite{C}, we need to have the following estimates:

\begin{itemize}
\item[$(i)$]  $\{P(w_\ve,v_\ve) + Q(w_\ve,v_\ve)\}_\ve$ is uniformly bounded in $L^q_{\rm loc}(\R_T)$, for some $q> 2$;

\item[$(ii)$]  $\{\ve \, (P_v v_{\ve x})_x\}_\ve$ is precompact in $H^{-1}_{\rm loc}(\R_T)$; 

\item[$(iii)$] $\{\ve \, (P_{w v} w_{\ve x} v_{\ve x} + P_{vv} v^2_{\ve x})\}_\ve$ is uniformly bounded in $L^1_{\rm loc}(\R_T)$; 

\item[$(iv)$]  $\{P_v \frac{1}{\ve} u^-_\ve\}_\ve$ is uniformly bounded in $L^1_{\rm loc}(\R_T)$. 
\end{itemize}

The conditions $(i)$, $(ii)$ and $(iii)$ are verified like in \cite{C}, by the previous estimates and the properties of $P$ and $Q$. 
Since $\{P_v(w_\ve, v_\ve)\}_\ve$  is uniformly bounded in $L^\infty_{\rm loc}(\R_T)$, the condition $(iv)$ is an 
easy consequence of \eqref{LInftyIneq}. Then, we can apply Murat's lemma to the sequence 
\begin{equation}
\label{SeqEntPair}
   \{(P(w_\ve,v_\ve))_t + (Q(w_\ve,v_\ve))_x \}_\ve, 
\end{equation}    
and \eqref{ConvV} follows like in \cite{C}, that is to say, the sequence \eqref{SeqEntPair} is uniformly bounded in $W^{-1,q}_\loc(\R_T)$, and also in
$\mathcal{M}(U)$ for any open bounded set $U \subset \R_T$, hence Murat's lemma follows and we are done. 
Here, $\mathcal{M}(U)$ denotes the space of finite Radon measures on $U$. Then, we obtain, as $\ve \to 0$, that 
$$
   w_{\ve} \to w, \quad v_{\ve} \to v \quad \text{strongly in  $L^p_\loc(\R_T)$
for $1\leq p < 2$,}
$$
and almost everywhere in $\R_T$, which is enough to show \eqref{ConvV}.

\medskip
Then, if we put 
$$
    u(t)= u_0 + \int_0^t v(\tau) \, d\tau, 
 $$
 we derive by \eqref{LInftyIneq} for each $R> 0$ that 
$$
   \int_0^t \!\! \int_{-R}^R u^- \, dx d\tau \leq \liminf_{\ve \to 0}  \int_0^t \!\! \int_{-R}^R u_\ve^- \, dx d\tau= 0, 
$$
and so $u \geq 0$ almost everywhere. 

\medskip
Moreover, with $\phi, \theta \in C^\infty_c(\R \times [0,T[)$, $\theta\geq 0$, we  
deduce from \eqref{SystReg} 
$$
\begin{aligned}
& 0 = \iint_{\R_T} (w_\ve \phi_t - v_\ve \phi_x) \, dxdt + \int_{\R} u_{0 x} \, \phi(x,0) \, dx,   
\\[5pt]
& \ve  \iint_{\R_T} v_{\ve xx} \theta \, dxdt =  \iint_{\R_T} \big(v_{\ve t} 
- (\sigma(w_\ve))_x - \frac{u_\ve^-}{\ve}  \big) \, \theta \, dxdt 
\end{aligned}
$$
and 
$$
\begin{aligned}
& \ve  \iint_{\R_T} v_{\ve xx} \theta \, dxdt = -  \ve  \iint_{\R_T} v_{\ve x} \theta_x \, dxdt, 
\\[5pt]
&\iint_{\R_T} \big(v_{\ve t} - (\sigma(w_\ve))_x  \big) \, \theta \, dxdt= - \iint_{\R_T} v_{\ve} \theta_t  \, dxdt
- \int_{\R} v_0 \theta(x,0) \, dx 
\\[5pt]
& \hspace{145pt} +  \iint_{\R_T} \sigma(w_\ve) \theta_x  \, dxdt. 
\end{aligned}
$$
Since for $\theta\geq 0$, it follows that 
$$
-\frac{1}{\ve} \iint_{\R_T}  u_\ve^- \, \theta  \, dxdt \leq 0,
$$
we derive 
$$
\begin{aligned}
-  \ve  \iint_{\R_T} v_{\ve x} \theta_x \, dxdt &\leq - \iint_{\R_T} v_{\ve} \theta_t  \, dxdt
- \int_{\R} v_0 \theta(x,0) \, dx 
\\[5pt]
&+  \iint_{\R_T} \sigma(w_\ve) \theta_x  \, dxdt
\end{aligned}
$$
and, passing to the limit when $\ve \to 0$, we obtain \eqref{SolOneWeak}. 
Therefore, under hypotheses $H1$-$H3$, we prove \eqref{DefCompOne01}-\eqref{SolOneWeak} of Definition~\ref{DefCompOne}.

\medskip
To obtain the continuity of the solution $u(x,t)$, we use the condition \eqref{continuity} in the energy estimate \eqref{EnergEst}. It is then clear that the sequence $\{u_\ve\}$ belongs to a bounded subset of $Y_2$, and the uniform convergence of $u_\ve$ to $u$ on compact sets is a consequence of Theorem \ref{MainThmReg}. This is sufficient to obtain \eqref{eqOComOne}, thus completing the proof of Theorem \ref{MainThmSol1} 
concerning the existence of complementary weak solutions. 
Indeed, since $u$ is continuous in $\R\times[0,T]$ and $u \geq 0$, the set 
$$
   U:= \{(x,t) \,; \, u(x,t)> 0\} \quad \text{is open.}
$$
Therefore, for each compact set $K \subset U$ there exisits a constant 
$$
\delta_K:= \min_{(t,x)\in K} u(x,t) >0.
$$
Now consider the family of approximate solutions
$u_\varepsilon$ of the penalized equation
\begin{equation}
\label{penalizedeq} 
   u_{\ve tt} - (\sigma(u_{\ve x}))_x
+ F_\varepsilon(u_\varepsilon)
= \varepsilon \, u_{\ve txx}.
\end{equation}
From \eqref{uniformconverg} in Theorem \ref{MainThmReg}, 
we have that $u^\varepsilon\to u$
uniformly on compact sets, hence there exists $\varepsilon_K> 0$, 
such that, for every $0<\varepsilon<\varepsilon_K$,
$$
  u_\varepsilon(x,t)\geq \frac{\delta_K}{2}> 0 \quad \text{in } K.
$$
By the definition of the penalization, it follows that
$$
   F_\varepsilon(u_\varepsilon)= 0 \quad \text{in $K,$}
$$
so that, locally in the free region $U$, the penalized equation coincides 
with the viscous equation without obstacle, that is to say,
$$
   u_{\ve tt} - (\sigma(u_{\ve x}))_x
= \varepsilon \, u_{\ve txx}, 
$$
Passing to the limit as $\varepsilon\to 0$ and using the arbitrariness of 
$K$, we obtain in distribution sense the limit equation, that is, 
\begin{equation}
\label{Eq_Limit}
u_{tt} - (\sigma(u_x))_x = 0, 
\quad \text{in $\{u>0\}$}. 
\end{equation} 
This proves the complementarity condition \eqref{eqOComOne}
rigorously. 

\medskip
Finally, we establish the second part of Theorem \ref{MainThmSol1}, which is the 
existence of variational weak solutions to the one obstacle string problem. 
As observed earlier, the variational formulation usually requires additional structure.
In the present setting, this requirement is precisely encoded in the following convergence
\begin{equation}
\label{StongL2V} 
   v_\ve \to v \quad \text{ strongly in  $L^2_\loc(\R \times \R_+)$}. 
\end{equation}
Recall that, due to estimate \eqref{EnergEst}, there exists a constant $C> 0$, such that, 
$$
   \sup_{t \geq 0} \|v_\ve(t)\|_{L^2(\R)} \leq C. 
$$
Therefore, passing to a subsequence, we have 
$$
   v_{\ve} \rightharpoonup v \quad \text{weakly in  $L^2(\R_T)$.}
$$
If we can show that, for each compact set $K \subset \R$ and for the same subsequence before, we have 
\begin{equation}
\label{NormconverV}
\text{$\|v_{\ve}\|_{L^2(K_T)}$ converges to $\|v\|_{L^2(K_T)}$ as $\ve \to 0$,}
\end{equation}
where $K_T= K \times [0,T]$, then with the above weak convergence, we obtain \eqref{StongL2V}. 

\smallskip
\begin{theorem} 
\label{ThmStongConvL2}
Under the above conditions, the norm convergence \eqref{NormconverV}
holds, and hence \eqref{StongL2V} follows. 
\end{theorem}

\begin{proof}
1. First, we recall the energy estimate \eqref{EnergEst},
that is,
\[
E_\ve(t):=\int_\R\Big(\tfrac12|v_\ve|^2+\Sigma(w_\ve)+\Phi_\ve(u_\ve)\Big)\,dx,
\]
where the penalization potential is given by $\Phi_\ve(u):=\frac{1}{2\ve}(u^-)^2$. Then, we obtain
\begin{equation}\label{Energy_global}
\frac{d}{dt}E_\ve(t)+\ve\|v_{\ve x}(t)\|_{L^2(\R)}^2=0,
\quad
E_\ve(t)+\ve\int_0^t\|v_{\ve x}(s)\|_{L^2(\R)}^2\,ds=E_\ve(0).
\end{equation}

\medskip
2. Now, we consider an energy estimate with spatial cutoff. Indeed, 
let the cutoff function $\chi\in C_c^\infty(\R)$, $0\le \chi\le 1$, 
such that, for any compact $K\subset\R$, $\chi\equiv1$ on $K$.
Then, we multiply \eqref{penalizedeq} by $\chi^2 v_\ve$ and integrate in $x$ to obtain 
\begin{equation}
\label{Energy_cut_pointwise}
\frac{d}{dt}E_\ve^\chi(t)
+\ve\int_\R \chi^2|v_{\ve x}|^2 \, dx
=
-2\int_\R \chi\chi'\sigma(w_\ve)v_\ve \, dx
+2\ve\int_\R \chi\chi'v_\ve (v_\ve)_x \, dx, 
\end{equation}
where analogously we have defined the localized energy
$$
E_\ve^\chi(t):= \int_\R \chi^2\Big(\tfrac12|v_\ve|^2+\Sigma(w_\ve)+\Phi_\ve(u_\ve)\Big) \, dx. 
$$
Next, we take a positive test function $\zeta \in C_c^1(0,T)$, and multiplying \eqref{Energy_cut_pointwise} by 
$\zeta(t)$ and integrating in $(0,T)$, we obtain 
\begin{align}
\label{Energy_cut_weaktime}
&-\int_0^T \zeta'(t) \, E_\ve^\chi(t) \,dt
+ \ve \int_0^T \zeta(t) \int_\R \chi^2|v_{\ve x}|^2 \, dx dt
\nonumber\\
&= -2 \int_0^T \zeta(t) \int_\R \chi \chi' \sigma(w_\ve) \, v_\ve \, dx dt
+ 2 \ve \int_0^T \zeta(t) \int_\R \chi \chi' \, v_\ve v_{\ve x} \, dxdt.
\end{align}
This identity is valid for every 
$\ve> 0$ and involves only absolutely convergent integrals.

\medskip
3. To follow, let us analyze the passage to the limit in each integral
of \eqref{Energy_cut_weaktime}. We consider  the following properties
\begin{itemize}
\item[(h1)] $v_\ve \rightharpoonup v$ in $L^2_{\mathrm{loc}}(\R_T)$ (weakly in $L^2(K_T)$),
already established. 

\smallskip
\item[(h2)] $w_\ve \to w$ strongly in $L^2_{\mathrm{loc}}(\R_T)$, and
$\Sigma(w_\ve) \to \Sigma(w)$ in $L^1_{\mathrm{loc}}(\R_T)$. We assume this for the moment.
Moreover, $\sigma(w_\ve)\to\sigma(w)$ in $L^1_{\mathrm{loc}}(\R_T)$. 

\smallskip
\item[(h3)] Since $u_\ve \to u$ uniformly on compact sets, already established, due to estimate \eqref{LInftyIneq}, 
and from the fact that the supremum of the sequence converges to the supremum (by uniform convergence), which is positive, we obtain 
$$
    \lim_{\ve \to 0} \iint_{\R_T} \chi^2 \, \Phi_\ve(u_\ve) \, dx dt= 0.
$$    
\end{itemize}
Therefore, we have 
\begin{itemize}
\item[(i)]  $\displaystyle \lim_{\ve \to 0} \int_0^T \zeta'(t) E_\ve^\chi(t) \, dt=  \int_0^T \zeta'(t) \bar{E}^\chi(t) \, dt$, where $\bar{E}^\chi$
is the weak-star limit of the energy $E_\ve^\chi$. 

\smallskip
\item[(ii)] We define the family of non-negative measures $\mu_\ve$ on $\R_T$ given by 
$$
   d\mu_\ve:= \ve \, |v_{\ve x}|^2 \, dx dt. 
$$
By the global energy estimate \eqref{Energy_global}, the total mass of $\mu_\ve$ is uniformly bounded, 
$$
   \iint_{\R_T} d\mu_\ve \leq \ve \int_0^T \|v_{\ve x}(s)\|_{L^2(\R)} ds \leq C. 
$$
Therefore, up to a subsequence, $\mu_\ve$ converges (weak-star convergence of measures) to a non-negative Radon measure $\mu$, that is
$$ 
     \lim_{\ve \to 0} \,  \int_0^T \zeta(t) \int_\R \ve \, \chi^2  \, |v_{\ve x}|^2 \, dx dt = \int_0^T \!\! \zeta  \int_\R \chi^2 \, d\mu= \mu[\zeta \chi^2], 
$$
where $\mu$ represents a defect measure that accounts for the concentration of energy dissipation as $\ve \to 0$. 

\smallskip
\item[(iii)] $\displaystyle \lim_{\ve \to 0} \int_0^T \zeta(t) \int_\R \chi \chi' \sigma(w_\ve) \, v_\ve \, dx dt= \int_0^T \zeta(t) \int_\R \chi \chi' \sigma(w) \, v \, dx dt$. 

\smallskip
\item[(iv)] Since $\displaystyle \ve \! \int_\R | \chi \chi' \, v_\ve \, v_{\ve x} | \, dx \leq \ve \, C(\chi) \, \|v_\ve(t)\|_{L^2(\R)} \, \|v_{\ve x}(t)\|_{L^2(\R)}$, it follows that 
$$
\ve\int_0^T \zeta(t) \int_\R \chi \chi' \, v_\ve \, v_{\ve x} \, dx dt \xrightarrow[\ve\to0]{} 0.
$$
\end{itemize}
Due to $(i)$-$(iv)$ we conclude from \eqref{Energy_cut_weaktime}, passing to the limit as $\ve \to 0$, that 
\begin{equation}
\label{Energia_Limite_Com_Defeito}
-\int_0^T \zeta'(t) \, \bar{E}^\chi(t) \, dt + \mu[\zeta \chi^2] = - 2 \int_0^T \zeta(t) \int_\R \chi\chi' \sigma(w)v \, dx \, dt. 
\end{equation}

\medskip
4. Let us show that $\mu= 0$. The limit equation is given by \eqref{Eq_Limit}, that is to say, 
$$
v_{t} - (\sigma(w))_x = 0, 
\quad \text{in $\{u>0\}$}. 
$$
Since $v$ does not have enough regularity in time to multiply the equation above, 
we consider, for $h> 0$ sufficiently small and $t \geq 0$, the so-called Steklov average
$$
   f_h(t):= \frac{1}{h} \int_t^{t+h} f(s) \, ds. 
$$ 
Thus applying the Steklov average to the equation \eqref{Eq_Limit}, we obtain 
$$ 
(v_h)_t - (\sigma(w))_{x,h} = 0,
$$
where we have used that, $(\sigma(w))_{x,h} \equiv (\sigma(w)_x)_h= (\sigma(w)_h)_x$
in distribution sense. Since $(v_h)_t \in L^\infty(0, T; L^2(\R))$, and thus $(\sigma(w))_{x,h}$, 
multiplying the above equation by $\chi^2 v_h$ and integrating in $x$, it follows that 
$$ 
   \int_\R \chi^2 (v_h)_t \, v_h \, dx - \int_\R \chi^2 \, (\sigma(w))_{x,h} \, v_h \, dx= 0. 
$$ 
Also $(w_h)_t= (v_h)_x$, then by the chain rule,  multiplying by 
$\zeta(t)$, integrating in $(0,T)$, and after integration by parts, we obtain 
\begin{equation}
\label{EqLimH}
   -\int_0^T \!\! \zeta'(t) \!\! \int_\R \chi^2 \Big( \frac{1}{2}v_h^2 + \Sigma(w)_h \Big) \, dxdt 
   = - 2 \int_0^T\!\! \zeta(t) \!\! \int_\R \chi \chi' \, \sigma(w)_h \, v_h \, dxdt + R(h),
 \end{equation}
 where we have defined
 $$ 
 R(h) \coloneqq \int_0^T \zeta(t) \int_\R \chi^2 \Big( \sigma(w)_h \partial_t w_h - \partial_t (\Sigma(w)_h) \Big) \, dxdt. 
$$ 
 Let us show that, $\lim_{h \to 0} R(h)= 0$. 
 Indeed, the integrand in $R(h)$ can be rewritten by exploiting the structure of the Steklov average, 
 that is, 
 $$
    \partial_t f_h = \frac{f(t+h)-f(t)}{h}, \quad \text{ for almost all $t$}. 
 $$   
 Therefore, we may write 
 $$ 
    \sigma(w)_h \partial_t w_h - \partial_t (\Sigma(w)_h) 
    = \sigma(w)_h \Big( \frac{w(t+h)-w(t)}{h} \Big) - \Big( \frac{\Sigma(w(t+h))-\Sigma(w(t))}{h} \Big).
$$ 
By Taylor's Theorem, we know that
 $$ 
    \Sigma(w(t+h)) - \Sigma(w(t)) = \sigma(w(t)) \cdot (w(t+h) - w(t)) + \frac{1}{2}\sigma'(\xi)(w(t+h) - w(t))^2. 
 $$
 Substituting in the expression of $R(h)$ and observing that, $\sigma(w)_h \to \sigma(w)$ strongly in $L^2$ as $h \to 0$, 
 the error term $R(h)$ is controlled by the modulus of continuity in time of $w$ in $L^2$. Since $w \in L^\infty(0,T; L^2(\R))$, 
 the Steklov average converges strongly, and the quadratic error term vanishes. Indeed, we have 
 $$ 
    |R(h)| \le C\,  \frac{1}{h} \! \int_0^{T-h}  \|w(t+h) - w(t)\|_{L^2}^2 \, dt \xrightarrow[h \to 0]{} 0. 
$$
 Therefore, passing to the limit as $h \to 0$ in \eqref{EqLimH}, we obtain 
  \begin{equation}
  \label{Energia_Ideal}
   -\int_0^T \zeta'(t) E^\chi(t) \, dt = -2 \int_0^T \zeta(t) \int_\R \chi \chi' \, \sigma(w) \, v \, dx dt, 
   \end{equation}
where 
$$ 
    E^\chi(t) = \int_\R \chi^2 \left( \frac{1}{2}v^2 + \Sigma(w) \right) dx, 
$$ 
that is, we recover the energy identity for the limit pair $(w, v)$.

\medskip
Comparing \eqref{Energia_Limite_Com_Defeito} with \eqref{Energia_Ideal}, 
we observe that the right-hand sides are identical, and thus subtracting these equations we get 
\begin{equation}
\label{Limit_Echi_distribution}
\int_0^T \zeta'(t) ( \bar{E}^\chi(t) - E^\chi(t) ) \, dt = \int_0^T \!\! \zeta  \int_\R \chi^2 \, d\mu. 
\end{equation}
Due to the lower semicontinuity of the $L^2$ norm and the strict convexity of $\Sigma$, we have 
$$
   \mathcal{F}(t):=  ( \bar{E}^\chi(t) - E^\chi(t) ) \geq 0. 
$$
Now, for each $\delta > 0$, we take $\zeta_\delta(s)$ as a regularization of the characteristic function of the interval 
$[\delta, t-\delta]$. In the distributional limit, it follows that, $\zeta_\delta' \rightharpoonup \delta_0 - \delta_t$. 
Applying this sequence of test functions to \eqref{Limit_Echi_distribution}, we obtain for almost every $t \in (0,T)$, 
$$
   \mathcal{F}(0) - \mathcal{F}(t)= \iint_{\R_T} \chi^2 \, d\mu. 
$$
Since the initial data are well-prepared, that is, at time $t= 0$ we have strong convergence of the norms ($F(0)= 0$), 
and also the Radon measure $\mu$ is non-negative, it follows that 
$$
  \mu= 0 \quad \text{and} \quad \mathcal{F}(t)= 0. 
$$

\medskip
5. The equality of the energies implies
$$ 
   \lim_{\ve \to 0} \iint_{\R_T} \chi^2 \left( \frac{1}{2}v_\ve^2 + \Sigma(w_\ve) + \Phi_\ve(u_\ve) \right) \, dx dt 
   = \iint_{\R_T} \chi^2 \left( \frac{1}{2}v^2 + \Sigma(w) \right) \, dx dt.
$$   
Using the strong convergence $w_\ve \to w$ (h2) and $\Phi_\ve \to 0$ (h3), we obtain
$$ 
\lim_{\ve \to 0} \iint_{\R_T} \chi^2 |v_\ve|^2 \, dx dt = \iint_{\R_T} \chi^2 |v|^2 \, dx dt. 
$$
Since $v_\ve \rightharpoonup v$ weakly in $L^2$, the convergence of the norms above give us the 
strong convergence $v_\ve \to v$ in $L^2_{\loc}(\R_T)$.

\medskip
6. Finally, it remains to show (h2). 
By the global energy estimate and the quadratic coercivity of $\Sigma$ (which follows from $\sigma'(\lambda)\ge c> 0$),
we have ${w_\ve}$ bounded in $L^2$ and ${\Sigma(w_\ve)}$ bounded in $L^1$ uniformly.
We already know that $w_\ve(x,t)\to w(x,t)$ for a.e.\ $(x,t) \in \R_T$, and also
by the global energy 
$$
  \int_\R \Sigma(w_\ve) \, dx \leq C.
$$ 
From condition \eqref{continuity}, $\Sigma(\lambda) \geq \frac{c}{2} |\lambda|^2$ and we have 
$$
\iint_{\R_T} |w_\ve|^2 \, dxdt  \leq C, 
$$
that is, $w_\ve$ is uniformly bounded in $L^2(\R_T)$ and the sequence $\{w_\ve\}$ is equi-integrable in $L^1(\R_T)$.
Indeed, by the de la Vall\'ee Poussin criterion, (see, for instance, Bogachev \cite{Bogachev}),
any sequence bounded in $L^q$ with $q>1$ is equi-integrable in $L^1$.
By the theory of Young measures (we refer to \cite{Malek}), there exists a measurable family of probability measures
$$
  \text{$\{\nu_{x,t}\}_{(x,t)\in \R_T}$ (the Young measure generated by $\{w_\ve\}$)}, 
$$
such that, for every continuous function
$f: \R \to \R$ with at most quadratic growth,
we have in distribution sense
$$
f(w_\ve)\rightharpoonup \langle \nu_{x,t},f\rangle
\quad \text{in $L^1(\R_T)$}, 
$$
that is to say, for every $\vp \in L^\infty(\R_T)$, 
$$
\lim_{\ve \to 0} \iint_{\R_T} \varphi(x,t)\, f(w_\ve(x,t)) \, dxdt
=\iint_{\R_T} \varphi(x,t)\,\langle \nu_{x,t},f\rangle \, dxdt. 
$$
Taking $f= \mathrm{id}$, we obtain
$\langle \nu_{x,t},\mathrm{id}\rangle = w(x,t)$ for almost all $(x,t)$, since $w_\ve \to w$ almost everywhere. 
By the convexity of $\Sigma$, Jensen's inequality yields for a.e. $(x,t) \in \R_T$,
\begin{equation}
\label{Jensen}
\langle \nu_{x,t},\Sigma\rangle \;\ge\; \Sigma\big(\langle \nu_{x,t},\mathrm{id}\rangle\big)
= \Sigma(w(x,t)).
\end{equation}
Integrating \eqref{Jensen} over $\R_T$, we conclude
\begin{equation}
\label{LiminfSigma}
\iint_{\R_T}  \Sigma(w) \, dxdt
\;\le\; \liminf_{\ve\to0}\iint_{\R_T}  \Sigma(w_\ve) \, dxdt.
\end{equation}
If we can show that, 
\begin{equation}
\label{LimsupSigma}
  \limsup_{\ve\to0} \iint_{\R_T}  \Sigma(w_\ve) \,dxdt \leq \iint_{\R_T} \Sigma(w) \, dxdt, 
\end{equation}
then jointly with \eqref{LiminfSigma} we obtain the strong convergence in $L^1$.
Equivalently, \eqref{LimsupSigma} enforces equality in \eqref{Jensen}, hence we have 
$$
    \langle \nu_{x,t},\Sigma\rangle=\Sigma(\langle \nu_{x,t},\mathrm{id}\rangle) \quad \text{for a.e. $(x,t) \in \R_T$}.
$$
Since $\Sigma$ is  strictly convex, the equality above occurs only when
$\nu_{x,t}$ is a Dirac measure, that is, $\nu_{x,t}= \delta_{w(x,t)}$ for a.e.\ $(x,t)$.
Consequently, there is neither oscillation nor concentration and, in particular,
$$
   w_\ve\to w \quad\text{strongly in $L^2_\loc(\R_T)$}.
$$
Moreover, 
$\sigma(w_\ve) \to \sigma(w)$ strongly in $L^1_\loc$, (see also \cite{C} p. 164, where this convergence is established). 
Indeed, due to \eqref{Sigmasigma} for each $\theta> 0$, there exists $R> 0$, such that, if $|\lambda|> R$, then 
$$
  |\sigma(\lambda)| \leq \theta \, \Sigma(\lambda). 
$$
Since $|\sigma(\lambda)| \leq C_\theta$ in $[-R, R]$ for some positive constant $C_\theta$, it follows that, 
$$
  |\sigma(\lambda)| \leq \theta \, \Sigma(\lambda) + C_\theta, \quad \text{for any $\lambda \in \R$}. 
$$
Taking $\theta> 0$ sufficiently small and knowing that $ \int_\R \Sigma(w_\ve) \, dx \leq C$, 
the sequence $\{\sigma(w_\ve)\}$ is equi-integrable in $L^1_\loc$ via the de la Vall\'ee Poussin criterion.
By Vitali's Convergence Theorem, the almost everywhere convergence $w_\ve \to w$
combined with this equi-integrability, we obtain
$$
   \sigma(w_\ve) \to \sigma(w) \quad \text{in $L^1_\loc(\R_T)$}. 
$$
Let us show \eqref{LimsupSigma}. 
From the (global) energy functional, for each $t \in [0,T]$,  
$$
E_\ve(t)= \int_\R \Big( \frac{|v_\ve|^2}{2} + \Sigma(w_\ve) + \Phi_\ve(u_\ve) \Big) \, dx.
$$
It follows from the energy dissipation \eqref{Energy_global} that, 
$E_\ve(t) \leq E_\ve(0)$. Therefore, we may write 
$$
    \int_\R \Sigma(w_\ve) \, dx + \int_\R  \frac{|v_\ve|^2}{2} \, dx \leq E_\ve(0) - \int_\R \Phi_\ve(u_\ve) \, dx. 
$$
Dropping the last integral on the right-hand side of the above inequality, 
integrating over the time interval $(0,T)$ and taking the limit superior as $\ve \to 0$, we have
$$
    \limsup_{\ve \to 0} \Big( \iint _{\R_T} \Sigma(w_\ve) \, dxdt + \iint_{\R_T} \frac{|v_\ve|^2}{2} \, dxdt \Big)   \leq T E(0), 
$$
where we used the fact that $E_\ve(0)$ converges strongly to $E(0)$.  
Using the property, $\limsup(a_n+b_n) \geq \limsup a_n + \liminf b_n$, we obtain 
$$
    \limsup_{\ve \to 0} \iint _{\R_T} \Sigma(w_\ve) \, dxdt + \liminf_{\ve \to 0} \iint_{\R_T} \frac{|v_\ve|^2}{2} \, dxdt  \leq T E(0). 
$$
By the lower semi-continuity of the $L^2$ norm, it follows that
$$
    \limsup_{\ve \to 0} \iint_{\R_T} \Sigma(w_\ve) \, dxdt  \leq \int_0^T \!\! E^1(t) dt - \iint_{\R_T} \frac{|v|^2}{2} \, dxdt = \iint_{\R_T} \Sigma(w) \, dxdt,
$$
where we have used \eqref{Energia_Ideal}, taking the limit as the cutoff $\chi \to 1$, and applied $\zeta_\delta$ as above, which proves \eqref{LimsupSigma}. 
\end{proof}

\medskip
Continuing with the second part of the proof of Theorem \ref{MainThmSol1},
let $\zeta(x,t)$ be a nonnegative, measurable and bounded function with 
$$
     \zeta_x \in L^\Sigma_\loc(\R_T), \quad \zeta_t \in L^2_\loc(\R_T).
$$     
Then, we have $\zeta \in C([0, T]; C_\loc^{0,\alpha}(\R))$, also as a consequence of Theorem \ref{MainThmReg}. 
Now, let $\theta \in C^{\infty}_c(\R \times [0,T))$ be a nonnegative test function.
Thus, we have 
\begin{equation}
\label{WeakEquatVF} 
\begin{aligned}
   &\iint_{\R_T} \big( u_{\ve tt} - (\sigma(u_{\ve x}))_x + F_\ve(u_\ve) \big) \, \theta \, (\zeta - u_\ve) \, dxdt
\\[5pt]
&= - \iint_{\R_T}  u_{\ve t} \, \theta_t \, (\zeta - u_\ve) \, dxdt 
- \iint_{\R_T}  u_{\ve t} \, \theta \, (\zeta_t - u_{\ve t}) \, dxdt
\\[5pt]
&- \int_{\R} v_0(x) \, \theta(x,0) (\zeta(x,0) - u_{0}(x)) \, dx + \iint_{\R_T} F_\ve(u_\ve) \, \theta \, (\zeta - u_\ve) \, dxdt
\\[5pt]
&+ \iint_{\R_T} \sigma(u_{\ve x}) \, \theta_x \, (\zeta - u_\ve) \, dxdt
+ \iint_{\R_T} \sigma(u_{\ve x}) \, \theta \, (\zeta_x - u_{\ve x}) \, dxdt. 
\end{aligned}   
\end{equation}
Passing to a subsequence, if necessary, we have 
$$
  \sigma(u_{\ve x}) \, u_{\ve x} \tobo{\ve \to 0} \sigma(u_{x}) \, u_{x} \quad \text{almost everywhere}, 
$$
and due to \eqref{EnergEst}, it follows that, 
$$
  \iint_{\R_T} \theta \, \sigma(u_{\ve x}) \, u_{\ve x} \, dxdt \leq C.
$$
Applying the Fatou's Lemma, we have $\theta \, \sigma(u_{x}) \, u_{x} \in L^1(\R \times ]0,T[)$, and 
\begin{equation}
\label{LimSigmaVF}
  \iint_{\R_T} \theta \, \sigma(u_{x}) \, u_{x} \, dxdt 
  \leq \liminf_{\ve \to 0}  \iint_{\R_T} \theta \, \sigma(u_{\ve x}) \, u_{\ve x} \, dxdt.
\end{equation}
Moreover, we have 
\begin{equation}
\label{ConvergPve}
\begin{aligned}
\iint_{\R_T} & F_\ve(u_\ve) \, \theta \, (\zeta - u_\ve) \, dxdt
= - \frac{1}{\ve} \iint_{\R_T} \theta \, u_\ve^- \, \zeta  \, dxdt
\\[5pt]
&+ \frac{1}{\ve} \iint_{\R_T} \theta \, (u_\ve^-)^2 \, dxdt \leq 2 \iint_{\R_T} \theta \, \Phi_\ve(u_\ve) \, dxdt \xrightarrow[\ve \to 0]{} 0. 
\end{aligned}
\end{equation}
From equations \eqref{WeakEquatVF}, \eqref{LimSigmaVF}, and \eqref{ConvergPve} under the convergence 
properties of $u_{\ve x}$ and $u_{\ve t}$, we conclude 
$$
\begin{aligned}
0 &\leq \limsup_{\ve \to 0} \iint_{\R_T} \big( u_{\ve tt} - (\sigma(u_{\ve x}))_x + F_\ve(u_\ve) \big) \, \theta \, (\zeta - u_\ve) \, dxdt
\\[5pt]
&\leq - \iint_{\R_T}  u_{t} \, \theta_t \, (\zeta - u) \, dxdt 
- \iint_{\R_T}  u_{t} \, \theta \, (\zeta_t - u_{t}) \, dxdt
\\[5pt]
&- \int_{\R} v_0(x) \, \theta(x,0) (\zeta(x,0) - u_{0}(x)) \, dx 
+ \iint_{\R_T} \sigma(u_{x}) \, \theta_x \, (\zeta - u) \, dxdt
\\[5pt]
& + \limsup_{\ve \to 0}  \iint_{\R_T} \sigma(u_{\ve x}) \, \theta \, (\zeta_x - u_{\ve x}) \, dxdt,
\end{aligned}
$$
from which we obtain 
\begin{equation}
\label{EqLimVF}
\begin{aligned}
0 &\leq - \iint_{\R_T}  u_{t} \, \theta_t \, (\zeta - u) \, dxdt 
- \iint_{\R_T}  u_{t} \, \theta \, (\zeta_t - u_{t}) \, dxdt
\\[5pt]
&- \int_{\R} v_0(x) \, \theta(x,0) (\zeta(x,0) - u_{0}(x)) \, dx 
\\[5pt]
&+ \iint_{\R_T} \sigma(u_{x}) \, \theta_x \, (\zeta - u) \, dxdt
+ \iint_{\R_T} \sigma(u_{x}) \, \theta \, (\zeta_x - u_{x}) \, dxdt.
\end{aligned}
\end{equation}
In particular, taking sequences of nonnegative test functions $\{\tilde{\theta}_n\}$, $\{\bar{\theta}_n\}$ 
that converge to 1, respectively, in $\R$ and $[0,T]$, with $\theta(x,t)= \tilde{\theta}_n(x) \, \bar{\theta}_n(t)$, it follows that
$$
\begin{aligned}
0 &\leq - \iint_{\R_T}  u_{t} \, \tilde{\theta}_n \bar{\theta}_{n t} \, (\zeta - u) \, dxdt 
- \iint_{\R_T}  u_{t} \, \tilde{\theta}_n \bar{\theta}_{n} \, (\zeta_t - u_{t}) \, dxdt
\\[5pt]
&- \int_{\R} v_0(x) \,\tilde{\theta}_n(x) \bar{\theta}_{n}(0)  ( \zeta(x,0) - u_{0}(x)) \, dx 
\\[5pt]
&+ \iint_{\R_T} \sigma(u_{x}) \,\tilde{\theta}_{n x} \bar{\theta}_{n} \, ( \zeta - u) \, dxdt
+ \iint_{\R_T} \sigma(u_{x}) \, \tilde{\theta}_n \bar{\theta}_{n} \, ( \zeta_x - u_{x}) \, dxdt.
\end{aligned}
$$
Applying the Dominated Convergence Theorem, we pass to the limit as $n \to \infty$ in the above inequality 
and obtain \eqref{EqVarOne}. Therefore, the proof of Theorem \ref{MainThmSol1} is finished. 

\begin{corollary}
Under the above convergences, the weak solutions satisfy, (in distribution sense), the entropy inequality in the free region, that is, 
$$
      \big(\eta(w,v)\big)_t +  \big(H(w,v)\big)_x \leq 0, \quad \text{in $\{u> 0\}$,}
$$
where $\eta$ and $H$ are given by \eqref{EntropyPair}. 
\end{corollary}

\begin{proof}
The corollary follows from \eqref{penalizedeq}, standard arguments in the theory of conservation laws, together with the preceding convergence results.
\end{proof} 
\begin{remark}
    Since we have the local uniform convergence of $u_\ve\to u$ in $\R_T$, with this regularisation-penalisation approximation we may interpret rigorously the third complementary condition in \eqref{LEqBis} as $\langle S[u],u\rangle=0$ in the duality sense of nonnegative bounded measures and continuous functions, locally in $\R_T$.
\end{remark}
\subsection{Two obstacles string problem}
\label{HQI}

The main issue of this section is to prove Theorem \ref{MainThmSol2}. 
Recall that, we need an additional assumption on $\sigma$, that is, 
condition \eqref{newcondsigma} stated in hypothesis $H4$.

\medskip  
Given $u_0^\ve \in H^3(\R) \cap H^1_\Sigma(\R)$ with $-1\leq u_0^\ve \leq 1$, $v_0^\ve \in H^2(\R)$, for $\ve> 0$, let 
$u_\ve(x,t)$ be solution of the following Cauchy problem,  
\begin{equation}
\label{Inequality}
\left \{
\begin{aligned}
    &u_{\ve tt} - \big( \sigma(u_{\ve x}) \big)_x + F_\ve(u_\ve)= \ve (u_{\ve t})_{xx},  \quad \text{in $\R_T$}, 
\\[5pt]
&u_\ve(x,0)= u_0^\ve(x), \quad u_{\ve t}(x,0)= v_0^\ve(x) \quad \text{in $\R$}, 
\end{aligned}
\right .
\end{equation}
where we recall that 
$$
    F_\ve(u_\ve)= - \frac{1}{\ve}\big((u_\ve+1)^{-} - (u_\ve-1)^{+} \big).
$$     
Again, the existence and uniqueness of a solution
of the Cauchy problem \eqref{Inequality}, with 
$$
    u_\ve \in C([0,T); H^3) \cap C^1([0,T); H^2) \cap C^2([0,T); L^2),
 $$    
 is a consequence of the Theorem 2.2
in \cite{C}. 
Similarly to the one obstacle problem, 
under the above conditions and notations, we have 
\begin{equation}
\label{FirstIneq} 
\begin{aligned}
 \frac{1}{2} \int_{\R} |u_{\ve t}|^2 \, dx &+ \int_{\R} \Sigma(u_{\ve x}) \, dx 
 \\[5pt]
&+  \int_{\R} \Phi_\ve(u_\ve) \, dx + \ve \int_0^t \! \int_{\R} (u_{\ve tx})^2\, dx= E_\ve(0), 
 \end{aligned}
\end{equation}
where the penalization potential $\Phi_\ve$ and $E_\ve(0)$ are given, respectively, by
$$
\begin{aligned}
  \Phi_\ve(u)&= \frac{1}{2 \, \ve} \Big( \big( (u -1)^+ \big)^2 + \big( (u +1)^- \big)^2 \Big), 
  \\[5pt]
E_\ve(0)&=  \int_{\R} \big( \frac{|v^\ve_{0}|^2}{2} + \Sigma(u^\ve_{0 x}) \big)\, dx \leq C. 
\end{aligned}
$$


Although, the $L^1$-estimate of the penalization term here is 
more delicate, since $F_\ve$ does not have fixed sign, we still have also a local $L^1$--estimate of the penalization term.
\begin{theorem}
\label{ThmL1LocTwo}

Under the conditions of Theorem \ref{MainThmSol2}, for each compact set $K \subset \R$, there exists $C> 0$ such that, for any $t \in [0,T]$, 
$$
   \int_0^t \!\! \int_K |F_\ve(u_\ve)| \, dxdt \leq C. 
$$
\end{theorem}
\begin{proof}
1. We adapt the argument used for the one obstacle case, (see item 2 of the proof of Theorem \ref{MainThm}), in order to handle the
two obstacle penalization. Indeed, 
we fix $K\Subset K'\Subset \R$, where $A\Subset B$ denotes that $A$ is compactly
contained in $B$.
Let $\varphi\in C_c^\infty(K')$ be a cutoff function such that
$0\le \varphi\le 1$, $\varphi\equiv 1$ on $K$, and
\[
|\varphi_x|\le C\,\delta^{-1},
\qquad
\delta:=\mathrm{dist}(K,\partial K')\in(0,1].
\]
For $R>0$ sufficiently large and any fixed $x_0\in K$, we define
\[
\varphi_R(x):=\varphi(x_0+x/R).
\]
Then $\varphi_R\equiv1$ on $K$ for $R$ large, $\supp\varphi_R\subset K'$, and
$|\partial_x\varphi_R|\le C(\delta R)^{-1}$.

\medskip
2. First, we estimate the lower obstacle.
Let $\zeta(s)$ be a Lipschtiz non-increasing function, given by
$$
\zeta(s)=
\left \{
\begin{aligned}
& \quad 1,  \hspace{30pt} s< -1,
\\
&\frac{-s + 1}{2},  -1\leq s\leq 1, 
\\
&\quad  0, \hspace{30pt}  s> 1,
\end{aligned}
\right .
$$ 
Then, we multiply equation $(\ref{Inequality})_1$ by
$\zeta(u_\ve) \, \varphi_R$ and integrate over $(0,t)\times\R$, for any $t>0$.
Proceeding similarly as in the one obstacle case, we obtain
\begin{equation}
\label{LowerObstacleEstimate}
\begin{aligned}
&\frac{1}{\ve}\int_0^t\!\!\int_{K'} (u_\ve+1)^- \,\zeta(u_\ve)\,\varphi_R\,dx\,d\tau
\leq \int_{K'} \big|\big(u_{\ve t} \, \zeta(u_\ve)\big)(t) - v_0 \, \zeta(u_0) \big| \, \vp_R \, dx
\\[5pt]
& + \frac{1}{2} \int_0^t \!\! \int_{K'} u_{\ve t}^2 \, \vp_R \, dx d\tau + \frac{1}{2} \int_0^t \!\! \int_{K'} |u_{\ve x} \, \sigma(u_{\ve x})| \, \vp_R \, dx d\tau
\\[5pt]
& +\frac{1}{R}  \int_0^t\!\!\int_{K'} |\sigma(u_{\ve x})| \, \zeta(u_\ve) \, |\vp^\prime| \, dx d\tau
+ \frac{\ve}{4} \int_{K'} \big| \big(u_{0 x} \big)^2 -  \big(u_{\ve x}(t) \big)^2 \big| \, \vp_R \, dx  
\\[5pt]
&+ \frac{\ve}{R}\int_0^t\!\!\int_{K'} |u_{\ve tx}| \, \zeta(u_\ve) \, |\vp^\prime| \,dx d\tau. 
\end{aligned}
\end{equation}
Since $(u_\ve+1)^-\,\zeta(u_\ve)= (u_\ve+1)^-$, and $|u_{\ve x} \, \sigma(u_{\ve x})| \leq \kappa \, \Sigma(u_{\ve x})$ by \eqref{newcondsigma},
using Cauchy--Schwarz inequality, the energy estimate \eqref{FirstIneq}, and the same
argument as in the one obstacle case to control
$\int_{K'}|\sigma(w_\ve)|\,dx$, there exists $C>0$ (independent of $\ve>0$), such that
$$
\frac{1}{\ve}\int_0^t \!\! \int_{K} (u_\ve+1)^-\,dx\,d\tau \le C.
$$
The estimate for the upper obstacle is obtained analogously. Now,
let $\zeta(s)$ be a Lipschtiz non-decreasing function, given by
$$
\zeta(s)=
\left \{
\begin{aligned}
& \quad 0,  \hspace{30pt} s< -1,
\\
&\frac{s + 1}{2},  -1\leq s\leq 1, 
\\
&\quad  1, \hspace{30pt}  s> 1,
\end{aligned}
\right .
$$ 
Testing $(\ref{Inequality})_1$ with $\zeta(u_\ve) \, \varphi_R$ and repeating the above
argument yields
$$
\frac{1}{\ve}\int_0^t \!\! \int_{K} (u_\ve-1)^+ \,dx d\tau \le C.
$$

\medskip
3. Finally, combining both estimates we conclude that
$$
\frac{1}{\ve}\int_0^t \!\! \int_K
\Big((u_\ve+1)^-+(u_\ve-1)^+\Big)\,dx\,d\tau
\le C,
$$
and therefore
$$
\int_0^t \!\! \int_K |F_\ve(u_\ve)|\,dx\,d\tau \le C,
$$
where $C>0$ is independent of $\ve>0$.
\end{proof}

\medskip
By proving that $\lim_\ve u_\ve= u$ almost everywhere, we may easily conclude by Theorem \ref{ThmL1LocTwo} that, 
for $t \geq 0$ and $x\in\R$, 
\begin{equation}
\label{Estim_u}-1\leq u \leq 1. 
\end{equation}
Indeed, following Lemma \ref{Lemma1} and since
$$
\begin{aligned}
   - \int_\R F_\ve(u_\ve) \, u_{\ve xx} \, dx&=  \frac{1}{\ve} \int_{\R} \big((u_\ve+1)^{-} - (u_\ve-1)^{+} \big) \,  u_{\ve xx} \, dx 
   \\[5pt] 
   &= - \frac{1}{\ve} \int_{\R} \big((u_\ve+1)^{-} - (u_\ve-1)^{+} \big)_x \,  u_{\ve x} \, dx \geq 0,  
\end{aligned}
$$
we obtain for each $t \geq 0$, the fundamental estimate (like Serre-Shearer), that is 
\begin{equation}
\label{FundamEst}
\begin{aligned}
&\ve \int_0^t  \!\! \int_{\R} \sigma^\prime(u_{\ve x}) \,  (u_{\ve xx})^2 \, dx dt + \ve^2 \int_{\R} (u_{\ve xx})^2(t) \, dx 
\\[5pt]
& \quad +  \int_{\R} (u_{\ve t})^2(t) \, dx + \ve \int_0^t \!\! \int_{\R} (u_{\ve tx})^2 \, dx dt \leq C \, (1+t).  
\end{aligned}
\end{equation}

Then, we may conclude (as in F. Caetano \cite{C}) that, there exist 
$$
   (w,v) \in \big( L^\infty([0,T);L^2(\R)) \big)^2, \quad w \in L^\Sigma_\loc(\R_T), 
$$ 
and a subsequence $\{(u_{\ve x}, u_{\ve t})\}_{\ve> 0}$, such that, passing to the limit as $\ve \to 0$, 
$$
   u_{\ve x} \to w, \quad u_{\ve t} \to v \quad \text{almost everywhere},
$$
weakly-$\ast$ in $L^\infty([0,\infty[;L^2(\R))$ and strongly in  $L^p_\loc(\R_T)$
for $1\leq p < 2$. Moreover, we obtain 
$$
  \sigma(u_{\ve x}) \to \sigma(w) \quad \text{in $L^1_\loc(\R_T)$}. 
$$
Now, with 
$$
  u(x,t)= u_0(x) + \int_0^t v(x,\tau) \, d\tau, 
$$
it follows that, $u_\ve \to u$ in $L^1_\loc(\R_T)$
and passing a subsequence, if necessary, 
$u_\ve \to u$ almost everywhere in $\R_T$, as $\ve \to 0$ and therefore we conclude \eqref{Estim_u}.


Once again, using the energy estimates for the approximating problem and applying Theorem \ref{MainThmReg} we obtain the continuity of the solution $u(x,t)$. Then, using the continuity of the limit solution $u$, 
and the uniform convergence of $u_\ve$ to $u$ on compact sets,
we obtain the complementarity condition \eqref{eqOComTwoBelow}--\eqref{eqOComTwo}. 
Indeed, since $u$ is continuous and $-1 \leq u \leq 1$, the set 
$$
   U:= \{(x,t) \,; \, -1< u(x,t)< 1 \} \quad \text{is open.}
$$
Thus, for each compact set $K \subset U$ there exists a constant 
$$
\delta_K:= \min_{(t,x)\in K} \{u(x,t) - (-1), 1 - u(x,t)\}> 0.
$$
Therefore, for each $(x,t) \in K$, we have 
$$
  -1 + \delta_K \leq u(x,t) \leq 1 - \delta_K. 
$$
Similarly, we consider the family of approximate solutions
$u_\varepsilon$ of the penalized equation
$$
   u_{\ve tt} - (\sigma(u_{\ve x}))_x
+ F_\varepsilon(u_\varepsilon)
= \varepsilon \, u_{\ve txx}.
$$
Again, using $u^\varepsilon\to u$
uniformly on compact sets, hence there exists $\varepsilon_K> 0$, 
such that, for every $0<\varepsilon<\varepsilon_K$,
$$
  | u_\varepsilon(x,t) - u(t,x) | \leq \frac{\delta_K}{2} \quad \text{in } K.
$$
Consequently, we have for each $(x,t) \in K$, 
$$
  -1 + \frac{\delta_K}{2} \leq u_\ve(x,t) \leq 1 -  \frac{\delta_K}{2}, 
$$
and by the definition of the penalization, it follows that
$$
   F_\varepsilon(u_\varepsilon)= 0 \quad \text{in $K.$}
$$
Therefore, locally in the free region $U$, the penalized equation coincides 
with the viscous equation without obstacle, that is,
$$
   u_{\ve tt} - (\sigma(u_{\ve x}))_x
= \varepsilon \, u_{\ve txx}.
$$
Finally, passing to the limit as $\varepsilon\to 0$ and using the arbitrariness of 
$K$, we obtain in distribution sense the limit equation
$$
u_{tt} - (\sigma(u_x))_x = 0, 
\quad \text{in $\{-1< u< 1\}$}, 
$$ 
which proves \eqref{eqOComTwo}, and thus the complementarity condition 
rigorously, since \eqref{eqOComTwoBelow}, \eqref{eqOComTwoAbove}
follows similarly observing the sign of $F_\ve$ in respectively domains.  

\medskip
Now, let us show \eqref{SolTwoWeakP}. In fact, \eqref{SolTwoWeakN} and \eqref{SolTwoWeak0} follow similarly. 
Then, taking a test function $\vp \in C^\infty_c(\R \times [0,T))$, $\vp \geq 0$, with ${\rm spt} \vp \cap \{t> 0\} \subset \{u< 0\}$,
we deduce from \eqref{Inequality}
$$
 \ve  \iint_{\R_T} u_{\ve txx} \, \vp \, dxdt =  \iint_{\R_T} \big(u_{\ve tt} 
- (\sigma(u_{\ve x}))_x + F_\ve(u_\ve)  \big) \, \vp \, dxdt 
$$
and 
$$
\begin{aligned}
& \ve  \iint_{\R_T} u_{\ve txx} \, \vp \, dxdt = -  \ve  \iint_{\R_T} u_{\ve tx} \,  \vp_x \, dxdt, 
\\[5pt]
&\iint_{\R_T} \big(u_{\ve tt} - (\sigma(u_{\ve x}))_x  \big) \, \vp \, dxdt= - \iint_{\R_T} u_{\ve t}  \, \vp_t  \, dxdt
- \int_{\R} v^\ve_0  \, \vp(x,0) \, dx 
\\[5pt]
& \hspace{145pt} +  \iint_{\R_T} \sigma(u_{\ve x}) \, \vp_x  \, dxdt, 
\\
&\iint_{\R_T} F_\ve(u_\ve) \vp \, dx dt
\\
& \quad = - \frac{1}{\ve} \iint_{\R_T} \big((u_\ve+1)^{-} - (u_\ve-1)^{+} \big) \vp \, dx dt= - \frac{1}{\ve} \iint_{\R_T} (u_\ve+1)^{-} \vp \, dx dt \leq 0, 
\end{aligned}
$$
where we have used the uniform convergence of $u_\ve$ to $u$ in $C([-R,R] \times [0,T])$, for any $R> 0$. 
Then, passing to the limit as $\ve \to 0$, 
we obtain \eqref{SolTwoWeakP}. Thereby completing the proof of Theorem \ref{MainThmSol2} 
concerning the existence of complementary weak solutions. 

\medskip
Finally, we establish the second part of Theorem \ref{MainThmSol2}, namely the existence of 
variational weak solutions to the two obstacles string problem. Again, the variational formulation
requires the following convergence
$$
u_{\ve t} \to u_t \quad \text{strongly in } L^2_{\loc}(\R_T), 
$$
which is established here as well by Theorem \ref{ThmStongConvL2}.

\medskip
Let $\zeta(x,t)$ be a function with $\zeta_x \in L^\Sigma_\loc(\R_T)$, $\zeta_t \in L^2_\loc(\R_T)$
and assume 
\begin{equation}
\label{Vtilde} 
 -1\leq \zeta\ \leq 1. 
\end{equation}
Again, taking $\theta \in C^{\infty}_c(\R \times [0,T))$, $\theta \geq 0$, we obtain 
$$
\begin{aligned}
&\iint_{\R_T} \big( u_{\ve tt} - ( \sigma(u_{\ve x}))_x + F_\ve(u_\ve) \big) \, \theta \,(\zeta - u_\ve) \, dx dt 
\\[5pt]
&= \ve  \iint_{\R_T}  (u_{\ve t})_{xx}  \, \theta \, (\zeta - u_\ve) \, dx dt= \alpha_\ve, 
\end{aligned}
$$
where 
$$
  \alpha_\ve= - \ve  \iint_{\R_T}  (u_{\ve t})_{x}  \, \theta_x \, (\zeta - u_\ve) \, dx dt
  - \ve  \iint_{\R_T}  (u_{\ve t})_{x}  \, \theta \, (\zeta - u_\ve)_x \, dx dt.
$$
Moreover, due to \eqref{FirstIneq} and \eqref{FundamEst}, we clearly obtain for each 
$T> 0$ fixed, 
$$
\begin{aligned}
  &\ve \Big(\iint_{\R_T} (u_{\ve tx})^2 \, dxdt \Big)^{1/2} \xrightarrow[\ve \to 0]{} 0, \quad  \text{and}
  \\[7pt]
  &\ve \iint_{\R_T} u_{\ve}^2 \, dxdt \leq C(T) \big( \|u_0\|^2_2 + \iint_{\R_T} u_{\ve t}^2 \, dxdt \big) \leq C_1(T). 
  \end{aligned} 
$$
Therefore, we conclude for each $\theta \in C^{\infty}_c(\R_T)$, $\theta \geq 0$, and
$$\zeta_x \in L^\Sigma_\loc(\R_T), \quad  \zeta_t \in L^2_\loc(\R _T),$$  with $-1\leq\zeta \leq 1$, that 
\begin{equation}
  \liminf_{\ve \to 0} \iint_{\R_T} \big( u_{\ve tt} - (\sigma(u_{\ve x}))_x + F_\ve(u_\ve) \big) \, \theta \, (\zeta - u_\ve) \, dxdt \geq 0, 
\end{equation}
where we have used the uniform convergence of $u_\ve$ to $u$ on compact sets. 
Then, we have 
\begin{equation}
\label{WeakEquat} 
\begin{aligned}
   &\iint_{\R_T} \big( u_{\ve tt} - (\sigma(u_{\ve x}))_x + F_\ve(u_\ve) \big) \, \theta \, (\zeta - u_\ve) \, dxdt
\\[5pt]
&= - \iint_{\R_T}  u_{\ve t} \, \theta_t \, (\zeta - u_\ve) \, dxdt 
- \iint_{\R}  u_{\ve t} \, \theta \, (\zeta_t - u_{\ve t}) \, dxdt
\\[5pt]
&- \int_{\R} v_0(x) \, \theta(x,0) (\zeta(x,0) - u_{0}(x)) \, dx + \iint_{\R_T} F_\ve(u_\ve) \, \theta \, (\zeta - u_\ve) \, dxdt
\\[5pt]
&+ \iint_{\R_T} \sigma(u_{\ve x}) \, \theta_x \, (\zeta - u_\ve) \, dxdt
+ \iint_{\R_T} \sigma(u_{\ve x}) \, \theta \, (\zeta_x - u_{\ve x}) \, dxdt. 
\end{aligned}   
\end{equation}

\medskip
Similarly to one obstacle problem, passing to a subsequence if necessary, it follows that 
$$
  \sigma(u_{\ve x}) \, u_{\ve x} \tobo{n \to \infty} \sigma(u_{x}) \, u_{x} \quad \text{almost everywhere}, 
$$
and due to \eqref{FirstIneq}, it follows that, 
$$
  \iint_{\R_T} \theta \, \sigma(u_{n x}) \, u_{n x} \, dxdt \leq C.
$$
Applying the Fatou's Lemma, we have $\theta \, \sigma(u_{x}) \, u_{x} \in L^1(\R_T)$, and 
\begin{equation}
\label{LimSigma}
  \iint_{\R_T} \theta \, \sigma(u_{x}) \, u_{x} \, dxdt 
  \leq \liminf_{n \to \infty}  \iint_{\R_T} \theta \, \sigma(u_{\ve x}) \, u_{\ve x} \, dxdt.
\end{equation}
Therefore, from equations \eqref{WeakEquat}  and \eqref{LimSigma}, we have 
\begin{equation}
\begin{aligned}
0 &\leq \limsup_{\ve \to 0} \iint_{\R_T} \big( u_{\ve tt} - (\sigma(u_{\ve x}))_x + F_\ve(u_\ve) \big) \, \theta \, (\zeta - u_\ve) \, dxdt
\\[5pt]
&\leq - \iint_{\R_T}  u_{t} \, \theta_t \, (\zeta - u) \, dxdt 
- \int_{\R}  u_{t} \, \theta \, (\zeta_t - u_{t}) \, dxdt
\\[5pt]
&- \iint_{\R_T} v_0(x) \, \theta(x,0) (\zeta(x,0) - u_{0}(x)) \, dx 
+ \iint_{\R_T} \sigma(u_{x}) \, \theta_x \, (\zeta - u) \, dxdt
\\[5pt]
&+ \limsup_{\ve \to 0}  \iint_{\R_T} \sigma(u_{\ve x}) \, \theta \, (\zeta_x - u_{\ve x}) \, dxdt,
\end{aligned}
\end{equation}
from which we obtain 
\begin{equation}
\label{EqLim}
\begin{aligned}
0 &\leq - \iint_{\R_T}  u_{t} \, \theta_t \, ( \zeta - u) \, dxdt 
- \iint_{\R_T}  u_{t} \, \theta \, (\zeta_t - u_{t}) \, dxdt
\\[5pt]
&- \int_{\R} v_0(x) \, \theta(x,0) (\zeta(x,0) - u_{0}(x)) \, dx
+ \iint_{\R_T} \sigma(u_{x}) \, \theta_x \, (\zeta - u) \, dxdt
\\[5pt]
&+ \iint_{\R_T} \sigma(u_{x}) \, \theta \, (\zeta_x - u_{x}) \, dxdt.
\end{aligned}
\end{equation}
Again, taking sequences of nonnegative test functions $\{\tilde{\theta}_n\}$, $\{\bar{\theta}_n\}$ 
that converge to 1, respectively, in $\R$ and $[0,T]$, we have from \eqref{EqLim} with $\theta(x,t)= \tilde{\theta}_n(x) \bar{\theta}_n(t)$, 
$$
\begin{aligned}
0 &\leq - \iint_{\R_T}  u_{t} \, \tilde{\theta}_{n} \bar{\theta}_{n t} \, (\zeta - u) \, dxdt 
- \iint_{\R_T}  u_{t} \, \tilde{\theta}_{n} \bar{\theta}_{n} \, ( \zeta_t - u_{t}) \, dxdt
\\[5pt]
&- \int_{\R} v_0(x) \, \tilde{\theta}_{n} \bar{\theta}_{n}(0)  ( \zeta(x,0) - u_{0}(x)) \, dx
+ \iint_{\R_T} \sigma(u_{x}) \, \tilde{\theta}_{n x} \bar{\theta}_{n} \, (\zeta - u) \, dxdt
\\[5pt]
&+ \iint_{\R_T} \sigma(u_{x}) \, \tilde{\theta}_{n} \bar{\theta}_{n} \, (\zeta_x - u_{x}) \, dxdt.
\end{aligned}
$$
Then, passing to the limit as $n \to \infty$, we obtain \eqref{EqVarTwo}. 
Consequently, the proof of Theorem \ref{MainThmSol2} is complete. 

\medskip

\begin{remark}
\label{DiffOneTwoPenal}
We observe that, unlike the one obstacle string problem, 
the penalization term now has no fixed sign and therefore cannot be globally discarded as we did before to obtain the
complementary weak solution for the one obstacle string problem. 
Therefore, with the uniform convergence we obtain the complementarity formulation localised in the neighbourhood of the lower and the upper obstacles. 
This shows that the two obstacles string problem, although similar, is not completely analogous to the one obstacle case and, in fact, 
it is requires additional work.
However, as in the one obstacle problem, the weak solutions satisfy, (in distribution sense), the entropy inequality in the free region when there is no contact with the obstacles, that is, 
$$
      \big(\eta(w,v)\big)_t +  \big(H(w,v)\big)_x \leq 0, \quad \text{in $\{-1< u<  1\}$,}
$$
where $\eta$ and $H$ are given by \eqref{EntropyPair}. This follows also from the previous convergences. 
\end{remark}

\section*{Data availability statement}
Data sharing is not applicable to this article as no data sets were generated or analysed during the current study.

 \section*{Acknowledgements}

The authors are indebted to Luc Tartar for valuable suggestions.

\smallskip
The authors J.P. Dias and J.F. Rodrigues were partially supported by the Funda\c{c}\~ao 
para a Ci\^encia e a Tecnologia (FCT) through the grant 
UID/04561/2025 - https://doi.org/10.54499/UID/04561/2025. 

\smallskip
The author W. Neves acknowledges and thanks the financial research support from CNPq
through the grant  313005/2023-0, 403675/2025-1, and also by FAPERJ 
(Cientista do Nosso Estado) through the grant E-26/204.171/2024.



\end{document}